\documentclass[10pt, makeidx]{amsart}

\DeclareMathAlphabet{\mathpzc}{OT1}{pzc}{m}{it}

\usepackage{amssymb}

\usepackage{enumerate}

\def\BT{\mathfrak{B}(\mathbb T)}

\newtheorem{theorem}{Theorem}[section]
\newtheorem{lemma}[theorem]{Lemma}

\newtheorem{proposition}[theorem]{Proposition}

\newtheorem{remark}[theorem]{Remark}

\newtheorem{definition}[theorem]{Definition}

\title{Agler-Commutant Lifting on an Annulus}

\author[S.~McCullough and S. ~Sultanic]
{Scott McCullough$^*$ and Saida Sultanic}

\address{Department of Mathematics\\
  University of Florida\\
  Box 118105\\
  Gainesville, FL 32611-8105\\
  USA}

\address{Sarajevo School Of Science and Technology\\
Bistrik 7\\
Sarajevo 71000 \\
Bosnia-Herzegovina}

\email{sam@math.ufl.edu}

\email{saida.sultanic@ssst.edu.ba}

\subjclass{47A20 (Primary), 47A48, 47A57, 47B32 (Secondary)}

\keywords{commutant lifting, Agler-Schur class, annulus }

\thanks{${}^*$ Research supported by NSF grants 0457504 and 0758306}

\makeindex

\begin{document}

\begin{abstract}
 This note presents a  commutant lifting theorem (CLT)
 of Agler type for the annulus $\mathbb A$.  Here the
 relevant set of test functions are the
 minimal inner functions on $\mathbb A$ - those
 analytic functions on $\mathbb A$
 which are unimodular on the boundary and have exactly
 two zeros in $\mathbb A$ - and the model
 space is determined by a distinguished member
 of the Sarason family of kernels over $\mathbb A$.
 The ideas and constructions borrow
 freely from  the CLT of
 Ball, Li, Timotin, and Trent \cite{BLTT}
 and Archer \cite{Archer} for the polydisc,
 and Ambrozie and Eschmeier for the ball in $\mathbb C^n$
 \cite{AE}, as well as
 generalizations of the de Branges-Rovnyak
 construction like found in Agler \cite{Ag}
 and Ambrozie, Englis, and M\"uller \cite{AEM}.
 It offers a template for extending the result result in \cite{MS}
 to infinitely many test functions.
 Among the needed new ingredients is the formulation of
 the factorization implicit in the statement
 of the results in \cite{BLTT}, \cite{Archer}
  and \cite{MS} in terms of certain
 functional Hilbert spaces of Hilbert space
 valued functions.
\end{abstract}

\maketitle


\thispagestyle{empty}

\section{Introduction}
 \label{sec:intro}
  Results going back to \cite{Ag} and including
  \cite{AM}, \cite{BT}, \cite{BTV},  \cite{Am} \cite{AEM}, \cite{AE}
  \cite{DMM}, \cite{DM} among others  view
  the starting point for  Agler-Pick interpolation
  as a
  collection of functions $\Psi,$ called test functions.
  \index{test functions} Roughly speaking
  one  constructs an operator algebra whose norm is as
  large as possible
  subject to the condition that each $\psi\in \Psi$
  is contractive. The corresponding Agler-Schur class,
  or $\Psi$-Agler-Schur class, is then the unit ball of this
  operator algebra and interpolation is within this class.

  The by now classical example is that of
  Agler-Pick interpolation in the $d$-fold polydisc
  $\mathbb D^d\subset \mathbb C^d$ with $\Psi=\{z_1,\dots,z_d\},$ where
  the $z_j$ are the coordinate functions \cite{Ag}\cite{AM}. In this case
  the unit ball of the  resultant
  operator  algebra of functions on $\mathbb D^d$
  is known  as the Agler-Schur class, often denoted
  $\mathcal S_d$. For $d=1,2$ this operator algebra
  is the same as $H^\infty (\mathbb D^d)$, but generally
  $\mathcal S_d$ and $H^\infty(\mathbb D^d)$
  are different. The literature contains many articles
  on the Agler-Schur class and its operator-valued generalizations.
  A sample of references include \cite{BB} \cite{BTV}\cite{Jury}\cite{AMb}.
  Of special relevance for this paper is the work
  of Ambrozie \cite{Am} and the subsequent articles
  \cite{DMM} and \cite{DM}, where the set of test functions $\Psi$
  is allowed to be infinite with a compact Hausdorff topology.

  It has long been known that Pick interpolation is
  a special case of commutant lifting  \cite{S}
   \cite{DMP} \cite{FF} \cite{Sz-N}.
  In this spirit Ball, Li,  Timotin, and
  Trent \cite{BLTT} formulate  and prove an Agler-Pick
  type commutant lifting theorem for the  polydisc.
  Significant
  refinements of both the statements and proofs of this result
  appear in the work of Archer \cite{Archer}.
  Ambrozie and Eschmeier \cite{AE} establish a related
  CLT for the unit ball in $\mathbb C^n.$ 
  In \cite{MS} we establish a generalization of these
  results to the case of a finite collection $\Psi$
  together with a distinguished reproducing kernel
  Hilbert space $H^2(k)$, unlocking the prior tight
  connection between the coordinate (test)
  functions $\{z_1,\dots,z_d\}$ and the  kernel
   $k$ for the  Hardy space
   $H^2(\mathbb D^d)$ in the case of the polydisc.
  In this more general context,
  the lack of an orthonormal basis explicitly
  expressible in terms of the test functions
  necessitated a number of innovations.

  In this article we pursue an Agler-Pick  type commutant lifting
  theorem with $\Psi$ the infinite collection of
  minimal inner functions on an annulus $\mathbb A$ - those with
  unimodular boundary values and exactly two zeros inside -
  and $H^2(k)$ a distinguished choice of Hardy Hilbert space
  on $\mathbb A$ - distinguished by the fact that $k(z,w)$
  is the only Sarason kernel for $\mathbb A$ which
  does not vanish for  $(z,w)\in \mathbb A\times \mathbb A$.
  In addition to certain measure theoretic considerations
  necessitated by the infinite collection of test functions,
  it also turns out that some structures not apparent or exploited
  in the case of finite test functions become important.
  We have borrowed
  freely from \cite{BLTT}, \cite{Archer}, \cite{AE}, 
  \cite{Ag} \cite{AEM} and of course \cite{MS}. 

  We thank the referee for many substantive suggestions which
  markedly improved the exposition. 

\section{Preliminaries and Main Result}
 Fix $0<q<1$ and let  $\mathbb A$ denote the annulus
 $\{z\in \mathbb C: q<|z|<1\}$.  The boundary of 
  the annulus comes in two parts, the outer boundary $B_0=\{|z|=1\}$
 and the inner boundary $B_1=\{|z|=q\}$.  As is customary, $\mathbb D$
  denotes the unit disc.  \index{outer boundary} \index{inner boundary}

\subsection{The test functions}
 \label{subsec:test-functions}
 The minimal inner functions \index{minimal inner function}
 on $\mathbb A$ are those (non-constant)
 analytic functions 
  $\phi:\mathbb A\to \mathbb D$ 
 whose boundary values are unimodular and 
 have the minimum number of zeros - two -  in $\mathbb A$. Up to canonical
 normalizations, they can be parametrized by the unit circle.

   If $\psi:\mathbb A \to \mathbb D$
 is a minimal inner function normalized by $\psi(\sqrt{q})=0$ 
  and $\psi(1)=1$, then 
  the second zero $w$ of $\psi$ must lie on the circle
 $\mathbb T=\{z:|z|=\sqrt{q} \}$ (see Section \ref{sec:more-test-functions}).
   Conversely, if $w$ is a point
 on this circle $\mathbb T$, then there is a (uniquely determined) 
 minimal inner function $\psi_w$  with $\psi_w(\sqrt{q})=0=\psi_w(w)$ 
 normalized by $\psi_w(1)=1$. In the case $w=\sqrt{q}$,
 this zero has multiplicity two.   Hence, 
 letting $\Psi=\{\psi_w:w\in\mathbb T\}\subset H^\infty(A)$, \index{$\Psi$}
 there is
 a canonical bijection $\mathbb T\to \Psi$ given by
 $w\mapsto \psi_w$ which turns 
 out to be a homeomorphism.

 For $z\in\mathbb A,$ let $E(z)$  denote the
 corresponding point evaluation on $\Psi$.
 Thus $E(z):\Psi\to\mathbb D$ is the  continuous function defined by
 $E(z)(\psi)=\psi(z).$ \index{$E(z)$}

\subsection{Transfer functions and the Schur class}
 \label{subsec:transfer}
   In the test function approach to interpolation and commutant lifting, 
   those functions built from the test functions
   as a transfer function of a unitary colligation play
   a key role and are known as 
   Agler-Schur class functions.  
   \index{unitary colligation}

 \begin{definition}\rm
  \label{def:colligation}
   A {\it $\Psi$-unitary colligation} is a tuple \index{unitary colligation}
   \index{$\Psi$-unitary colligation}
  $\Sigma=(\rho,A,B,C,D,\mathcal E,\mathcal H)$ where
 \begin{itemize}
  \item[(i)] $\mathcal E$ and $\mathcal H$ are Hilbert spaces;
  \item[(ii)] $\rho:C(\mathbb T)\to \mathcal B(\mathcal E)$ is
    a unital representation; and
  \item[(iii)] the block operator
    \begin{equation*}
      U=\begin{pmatrix} A & B\\ C& D\end{pmatrix}:
       \begin{matrix} \mathcal E \\ \oplus \\ \mathcal H \end{matrix} \to
        \begin{matrix} \mathcal E \\ \oplus \\ \mathcal H \end{matrix}
    \end{equation*}
   is unitary.
 \end{itemize}

   The corresponding {\it transfer function} \index{transfer function}
   is the function
   on $\mathbb A$ with values in $\mathcal B(\mathcal H)$ given
   by
  \begin{equation*}
   W_\Sigma = D+ C(I-Z A)^{-1}ZB,
  \end{equation*}
   where $Z:\mathbb A\to \mathcal B(\mathcal E)$ is the function
   $\rho(E(z))$.
\end{definition}

  The collection $\mathcal S(\mathbb A,\mathcal H)$ of
  functions $F:\mathbb A\to \mathcal B(\mathcal H)$ with a
  transfer function representation is called the
  Schur-Agler class\index{Schur-Agler class}.
  It coincides with the usual unit ball of
  $H^\infty(\mathbb A)$ for scalar-valued functions \cite{DM}
  $\mathcal H=\mathbb C$).
  We believe that, using Agler's rational dilation
  theorem \cite{Ag} and arguments like those in
  \cite{DM} or those of \cite{DMrat}, the same is true for operator-valued
  $H^\infty(\mathbb A),$ but postpone further consideration
  of this issue.

\subsection{A Hardy space of the annulus}
 \label{subsec:hardy-space}
  Results of Sarason \cite{S}, Abrahamse and Douglas \cite{AD}, and 
  Abrahamse \cite{Ab} among others identify a certain
  one parameter family of Hardy Hilbert spaces over the annulus
  which, collectively, play  the same role for $\mathbb A$  as 
  the classical Hardy space plays for $\mathbb D$. 

  For $t>0$, let $\mu_t$
  denote the measure on the boundary of $\mathbb A$ which is
  the usual normalized arclength measure on the outer boundary
  $B_0$ (so that $\mu_t(B_0)=1)$,
  but is $t$ times normalized  arclength measure on the inner boundary
  $B_1$ (so that $\mu_t(B_1)=t$).
  Let $H^2_t=H^2_t(\mathbb A)$ denote the Hardy
  Hilbert space obtained by closing up functions analytic in
  a neighborhood of the closure of $\mathbb A$ in $L^2(\mu_t)$.

  It is straightforward to check that the set
 \begin{equation}
  \label{eq:zeta-basis}
   \zeta_n= \frac{z^n}{\sqrt{1+tq^{2n}}}, \ \ \ n\in\mathbb Z,
 \end{equation}
   is an orthonormal basis for $H^2_t$. In particular,
 \begin{equation}
  \label{eq:defkt}
   k(z,w;t)=\sum_{n\in\mathbb Z} \frac{(zw^*)^n}{1+tq^{2n}}
 \end{equation}
  is the reproducing kernel for $H^2_t$.

   Each $\varphi \in H^\infty(\mathbb A)$ determines an operator $M_t(\varphi)$
   of multiplication by $\varphi$ on $H^2_t$ whose adjoint satisfies
 \begin{equation*}
   M_t(\varphi)^* k(\cdot,w;t)=\varphi(w)^* k(\cdot,w;t).
 \end{equation*}
 
   From equation \eqref{eq:defkt}, it is evident that
   $U: H^2_{q^2t} \mapsto H^2_{t}$ given by
   $U f= zf$ is unitary. It also intertwines  $M_{tq^2}$ and $M_t$; i.e.,
   $UM_{tq^2}(\varphi)=M_t(\varphi)U$. Modulo this equivalence,
   the collection $(H^2_t,M_t)$ is a family of representations
   of $H^\infty(\mathbb A)$  parametrized by the unit
   circle.  Up to unitary equivalence, these
   are Sarason's Hardy spaces of the annulus 
  \cite{S} that appear in \cite{Ab}.  They
   are also, over $\mathbb A$,  the rank one bundle shifts of
   Abrahamse and Douglas \cite{AD}.

   The kernel functions $k(z,w;t)$ have theta function representations from
   which the proposition below follows. 
   From here on, let $k(z,w)=k(z,w;1)$ and $H^2(\mathbb A)=H^2_{1}(\mathbb A)$.
   \index{$k(z,w)$} This is
   our distinguished Hardy space and its kernel. Set $k_w(z)=k(z,w).$

 \begin{proposition}
  \label{prop:ourk}
    The kernel  $k(\cdot,\cdot)$ doesn't vanish in the annulus; i.e.,  for
    $z,w\in \mathbb A$,  $k(z,w)\neq 0$, but it does vanish on the
    boundary as $k(1,-1)=0$.  Further, there is a constant $C^\prime>0$
    independent of $z$ and $w$ in $\mathbb A$ so that
  \begin{equation*}
     \frac{1}{k(z,w)} = C^\prime k(z,-w).
  \end{equation*}

   If  $t\neq q^{2m}$ (for any $m$),
  then there exists $z,w\in\mathbb A$ such
   that $k(z,w;t)=0$.
 \end{proposition}

  A proof of the proposition appears in Section \ref{sec:morek}.

  In the sequel, frequent use will be made of the Hilbert
 space tensor product $H^2(k)\otimes \mathcal H$, where $\mathcal H$
  is itself a Hilbert space. A convenient way to
 define this Hilbert space is as those (Laurent) series
\begin{equation*}
  h=\sum_{j\in\mathbb Z} \zeta_j \otimes h_j,
\end{equation*}
 for which $\sum \|h_j\|^2$ converges. The inner product is defined by
\begin{equation*}
  \langle h,g\rangle =\sum \langle h_j,g_j\rangle.
\end{equation*}
  For $z\in\mathbb A$, the sum
\begin{equation*}
  h(z) =\sum \zeta_j(z)\otimes h_j
\end{equation*}
  converges absolutely.
  It follows that, for a fixed $g\in\mathcal H$,
\begin{equation*}
   \langle h(z),g\rangle_{\mathcal H} = \langle h,k_z\otimes g\rangle.
\end{equation*}

  A function $W:\mathbb A\to \mathcal B(\mathcal H)$ defines a contraction
  operator $M_W$ on $H^2(k)\otimes\mathcal H$ by
 \begin{equation}
   \label{eq:mult-op}
   M_W^* [k_z\otimes g] = k_z \otimes W(z)^*g
 \end{equation}
  if and only if the (operator-valued) kernel
 \begin{equation*}
  \mathbb A\times \mathbb A \ni (z,w) \mapsto
     (I-W(z)W(w)^*)k(z,w)
 \end{equation*}
  is positive semi-definite \cite{AM3}\cite{BBx}.   Because, 
  for $h\in H^2(k)\otimes \mathcal H,$ 
 \begin{equation*}
    \langle M_W h, k_z\otimes g\rangle =
       \langle W(z)h(z),g\rangle,
 \end{equation*}
  it is natural to write $(M_W h)(z)=W(z)h(z)=(Wh)(z)$ to denote the operator
  $M_W$ and identify it with the function $W(z)$.

  The following standard lemma will be used often and 
  without comment in the sequel.

 \begin{lemma}
  \label{lem:bounded-pointwise-weak}
    If $W_n:\mathbb A\to \mathcal B(\mathcal H)$ is a sequence
    of functions which converge pointwise (in the norm topology)
    to $W$
    and if $\{W_n\}$ is uniformly bounded, then $M_W$
   is bounded  and the sequence
   $(M_{W_n})$ converges WOT to $W$.
 \end{lemma}

 For expository purposes, we record the following
 nice relation between the kernel $k$ and the test functions.

\begin{proposition}[\cite{MHA,Mtri}]
 \label{prop:testandkernel}
   For the test function $\psi$ with zeros $\sqrt{q}$ and
   $w$ with $|w|=\sqrt{q}$ the kernel 
 \begin{equation*}
   \mathbb A\times \mathbb A \ni (z,w)\mapsto   k(z,w)(1-\psi(z) \psi(w)^*) =
     \langle (I-M_{\psi}M_{\psi}^*)k(\cdot,w),k(\cdot,z) \rangle
 \end{equation*}
   has rank two and is positive semi-definite. 

   Further, $M_{\psi}$ is a shift of multiplicity two and
   the kernel of $I-M_{\psi}M_{\psi}^*$ is the span
   of $k(\cdot,\sqrt{q})$ and $k(\cdot, w)$
   (except of course when $w=\sqrt{q}$ when we must resort to
    using a derivative).
\end{proposition}

\subsection{Some representations and the functional calculus}
 \label{subsec:reps}
  Let $T$ denote an operator on a Hilbert space $\mathcal M$
  with $\sigma(T)\subset \mathbb A$. This spectral condition
  (as opposed to the more liberal $\sigma(T)\subset
  \overline{\mathbb A}$)
  is imposed because we wish to consider
  $\frac{1}{k}(T,T^*)$ and $\frac{1}{k}$ does not extend
  to be analytic in $z$ and $w^*$ beyond $\mathbb A\times\mathbb A.$
  Let $T$ also denote the corresponding representation
  $T:H^\infty(\mathbb A)\to \mathcal B(\mathcal M)$,
  given by $T(f)=f(T)$. We also use the notation $T_f=f(T)$.
  Note that $T$ is weakly continuous in the
  sense that if $f,f_n \in H^\infty(\mathbb A)$ and
  $f_n$ converges to $f$ uniformly on compact sets,
  then $T_{f_n}$ converges in operator norm to $T_f$.

\subsubsection{The hereditary functional calculus}
    Given an operator $T$ 
  and a polynomial $p(z,w)=\sum p_{j,\ell}z^j (w^*)^\ell,$
  the hereditary \index{hereditary calculus}
  calculus of Agler \cite{Ag} evaluates 
  $p(T,T^*)= \sum p_{j,\ell} T^j(T^*)^\ell.$
   The calculus extend to functions $f(z,w)$ which are
  analytic in $z$ and coanalytic in $w$ on a neighborhood
  of $\sigma(T)\times \sigma(T)^*$. Here we will not
  need the full power of the calculus, but we do need
  a generalization like that found in \cite{AEM}.
  For integers $j$, 
  let $T_j$ denote $T_{\zeta_j},$ where $\zeta_j$
  is defined in equation \eqref{eq:zeta-basis} (with $t=1$).

  For an operator $T\in\mathcal B(\mathcal M)$ with
  $\sigma(T)\subset\mathbb A,$
  and $G\in\mathcal B(\mathcal M)$, the sum
 \begin{equation*}
   k(T,T^*)(G) := \sum_{-\infty}^{\infty} T_j G T_j^*
 \end{equation*}
   converges absolutely.  The same is also true  of
 \begin{equation*}
  \frac{1}{k}(T,T^*)(G) := C^\prime \sum_{-\infty}^{\infty} (-1)^j T_j G T_j^*.
 \end{equation*}

  The following Lemma follows from the functional calculus
  considerations in \cite{AEM} together with the fact 
  that, by hypothesis,  
  $\sigma(T)\times \sigma(T^*)\subset \mathbb A\times \mathbb A$
 (see \cite{E}). 

\begin{lemma}
 \label{lem:overk}
   Let $T,G\in\mathcal B(\mathcal M)$ be given.
   If $\sigma(T)\subset \mathbb A$, then
 \begin{equation*}
   k(T,T^*)(\frac{1}{k}(T,T^*)(G))=G,
 \end{equation*}
  and likewise,
 \begin{equation*}
   \frac{1}{k}(T,T^*)(k(T,T^*)(G))=G.
 \end{equation*}

  If $G_\alpha\in\mathcal B(\mathcal M)$ is a (norm bounded) net
  which converges WOT to $G\in\mathcal B(\mathcal M)$, then $k(T,T^*)(G_\alpha)$
  converges WOT to $k(T,T^*)(G)$; and likewise
  $\frac{1}{k}(T,T^*)(G_\alpha)$ converges WOT to
  $\frac{1}{k}(T,T^*)(G)$
\end{lemma}


\subsection{The model operator}
 \label{subsec:model-operator}
   The operator of multiplication by $z$ on $H^2(k)$ gives
   rise to the representation
   $M:H^\infty(\mathbb A)\to \mathcal B(H^2(k))$
   \index{$M$, model operator}
   defined by $M(f)g=M_f g= fg$.  (Note $\sigma(M_\zeta)=\overline{\mathbb A}$.)
   To simplify notation, if $\mathcal H$
   is a Hilbert space, we also use $M$ to denote the representation
   $M\otimes I_{\mathcal H}$ on $H^2(k)\otimes \mathcal H$.

   We say that $M$ on $H^2(k)\otimes \mathcal H$
   {\it lifts} \index{lifts} the representation
   $T:H^\infty(\mathbb A) \to \mathcal B(\mathcal M)$ if there is an isometry
   $V:\mathcal M\to H^2(k)\otimes \mathcal H$ so that $VT^*=M^*V$; i.e.,
   for each $f\in H^\infty$, $VT_f^*=M_f^*V$.  
   An application of Runge's Theorem, or simply arguing with 
   Laurent series, together with the considerations
   in Subsection \ref{subsec:reps} shows that it suffices to assume
   that $VT_{\zeta}^*=M_\zeta^* V$. 

   If $\mathcal M\subset H^2(k)\otimes \mathcal H$ is invariant for
   $M^*$ (that is $M_f^* \mathcal M\subset \mathcal M$ for all
   $f\in H^\infty(\mathbb A)$), then $T= V^* M V$ given by
   $T_f=PM_fP$, where $V$ is the inclusion of  $\mathcal M$
   into $H^2(k)\otimes\mathcal H$, 
   is also a representation. Indeed, in this case $M$ lifts $T$.

\subsection{Agler decompositions}
 \label{subsec:Agler-SoS}
  Suppose $T\in\mathcal B(\mathcal M)$
  is an operator with $\sigma(T)\subset \mathbb A$
  and such that $T$ is lifted by $M.$
  Further suppose $X\in\mathcal B(\mathcal M)$ commutes with
  $T$; i.e., $T_fX=XT_f$ for all $f\in H^\infty(\mathbb A)$.
  As in Subsection \ref{subsec:model-operator}, note that
  it suffices to assume that $T_\zeta X = XT_\zeta$.

  An {\it Agler decomposition},  \index{Agler decomposition}
    for the pair $(T,X)$ is a $\mathcal B(\mathcal M)$-valued measure
  $\mu$ on $\BT$, the Borel subsets of $\mathbb T$ (identifying $\Psi$
  with $\mathbb T$), $\mu:\BT \to \mathcal B(\mathcal M)$ such that
 \begin{itemize}
  \item[(i)] for each $\varphi$ in the scalar Schur class and each
   Borel set $\omega$,
  \begin{equation}
   \label{eq:psi-L-bounded}
    k(T,T^*)(\mu(\omega))- T_\varphi k(T,T^*)(\mu(\omega)) T_\varphi^* \succeq 0
     \mbox{ and;}
  \end{equation}
  \item[(ii)]
 \begin{equation}
  \label{eq:key-mu}
    \frac{1}{k}(T,T^*)(I-XX^*)
     = \mu(\mathbb T) - \int T_\psi d\mu(\psi) T_\psi^*.
 \end{equation}
 \end{itemize}

  Here, for self-adjoint operators $A$ and $B$, the notation $A\succeq B$
  means $A-B$ is  positive semi-definite
  and similarly $A\succ B$ means $A-B$ is positive definite.

  Several remarks are in order.

\begin{remark}\rm
  The integral on the right hand side  of item (ii)
  is interpreted weakly as follows. Given
  a measurable partition $P=(\omega_j)_{j=1}^n$ of
  $\mathbb T$ and points $S=(s_j\in\omega_j)$, let
  $\Delta(P,S,\mu) = \sum T_{s_j} \mu(\omega_j)T_{s_j}^*$.
  The tagged partitions $(P,S)$ form an directed set 
  ordered by refinement of partitions, 
  and it turns out, because of \eqref{eq:psi-L-bounded}, that
  the net $\{\Delta(P,S,\mu):(P,S)\}$  converges in the WOT
  and its limit is the integral.  

  Thus the integral here, and the 
  corresponding $L^2$ spaces that appear later,
  shares much with the  integration theory based
  of Riemann sums and is not so different than others
  found in the literature.  For a recent example,
  see \cite{GWZ}.   Detail of
  the construction are given in Section \ref{sec:funhilby}.
  Narrowly tailoring the development to the present needs
  has the virtue of keeping the presentation self contained 
  and ultimately the paper shorter. 
 \end{remark}

\begin{remark}\rm
 \label{rem:Lambda}
  The definition of operator-valued measure requires $\mu$
  to be WOT countably additive. Thus, the second
  part of Lemma \ref{lem:overk}
  implies that $\Lambda(\omega)= k(T,T^*)(\mu(\omega))$
  is also an operator-valued measure.

  It is not assumed that $\mu(\mathbb T)=I$.
\end{remark}

\subsection{The main result}
 \label{subsec:main}
 
\begin{definition}\rm
 \label{def:minimal-lifting}
  Given $T\in\mathcal B(\mathcal M)$ with $\sigma(T)\subset \mathbb A$,
 a lifting $VT^*=M^*V$ of $T$ by $M$ on $H^2(k)\otimes \mathcal H$
  is {\it minimal}
 if $Q^*V \mathcal M$ is dense in $\mathcal H$.
  Here $Q^* \sum f_j \zeta_j = f_0$.
\end{definition}

  In the next section it is shown that a minimal lifting is
  essentially unique.  The following theorem is the 
 main result of this paper.

\begin{theorem}
 \label{thm:main}
  Let $\mathcal M$ be a separable Hilbert space.
  Suppose $X,T\in\mathcal B(\mathcal M)$ and
 \begin{itemize}
   \item[(i)] $\sigma(T)\subset \mathbb A$;
   \item[(ii)] 
      $M$ on
      $H^2(k)\otimes \mathcal H$ with $VT^*=M^*V$ is a minimal
      lifting; and
   \item[(iii)] $XT_\varphi=T_\varphi X$ for each $\varphi\in H^\infty(\mathbb A).$ \end{itemize}
   The following are equivalent.
 \begin{itemize}
  \item[(sc)] There is an $F\in\mathcal S(\mathbb A, \mathcal H)$ so that
              $XV^*=V^*M_F$.
  \item[(ad)] There is an Agler decomposition
    $\mu:\BT \to \mathcal B(\mathcal M)$
    for the pair $(T,X).$
 \end{itemize}
\end{theorem}

\begin{remark}\rm
 It is illuminating to consider the special case of Agler-Pick interpolation on 
 $\mathbb A$. 
 Let $z_1,\dots,z_n\in\mathbb A$ and $w_1,\dots,w_n\in\mathbb D$ be given.
 Let $\mathcal M\subset H^2(k)$ denote the span of $\{k_{z_j} \}$ and
 let $V$ denote the inclusion of $\mathcal M$ into $H^2(k)$.  Then
 $T$ defined by $T=V^*M V$ is lifted by $M$ and its spectrum is the set
 of $\{z_j\}$.  Define $X^*$ on $\mathcal M$ by $X^* k_{z_j} = w_j^* k_{z_j}$.
  Then $X$ commutes with $T$. In this case 
\[
 \langle \frac{1}{k}(T,T^*)(I-XX^*) k_{z_\ell},k_{z_j} \rangle
  = 1-w_j w_\ell^*
\]
 and 
\[
   \int T_\psi \langle d\, \mu(\psi) T_\psi^* k_{z_\ell},k_{z_j}\rangle 
    = \int \psi(z_j)\psi(z_\ell)^* 
      \langle d\, \mu(\psi) k_{z_\ell},k_{z_w}\rangle.
\] 
  Thus part (ii) in an Agler decomposition takes the form, 
\[
 1-w_j w_\ell^* = \int [1-\psi(z_j)\psi(z_\ell)^*] 
  \langle d\, \mu(\psi) k_{z_\ell},k_{z_w}\rangle. 
\]
\end{remark}

\section{More on Liftings}
 \label{sec:liftings}
   Recall the orthonormal basis $\{\zeta_n\}_{n\in\mathbb Z}$ (with $t=1$)  of
   equation (\ref{eq:zeta-basis}) and let $T_j$ and $M_j$ denote
   $T_{\zeta_j}$ and $M_{\zeta_j}$ respectively,
   where $M^*$ acting on $H^2(k)\otimes \mathcal H$ lifts
   $T^*$ acting on $\mathcal M$.

  The following is a version of a theorem of Ambrozie, Englis, and M\"uller \cite{AEM},
  a result very much in the spirit of the de Branges-Rovnyak construction
  \cite{dBR} and related to the results of \cite{Ag-coanalytic}. 

\begin{proposition}
 \label{prop:AEM}
    Suppose $T\in \mathcal B(\mathcal M)$
   and $\sigma(T)\subset \mathbb A$.

    If $M=M\otimes I_{\mathcal H}$ lifts $T$ with $VT^*=M^*V$, then
 \begin{equation}
  \label{eq:mustV}
     Vh =\sum \zeta_j \otimes R T_j^* h
 \end{equation}
   where $R=Q^*V:\mathcal M\to\mathcal H$,  the operator
   $Q:\mathcal H\to H^2(k)\otimes \mathcal H$
   is defined by  
\[
   Q^* \sum f_j\otimes\zeta_j = f_0,
\] 
   and  the sum converges in norm. In particular,
   the (non-decreasing) sum
\begin{equation}
 \label{eq:sumstoI}
    \sum_{j=-n}^n  T_j R^*R T_j^*
\end{equation}
   converges WOT to the identity.

   Conversely, if there is an $R:\mathcal M\to \mathcal H$ so that the
   sum  in equation (\ref{eq:sumstoI}) converges WOT to the identity,
   then $M$ lifts $T$ via $VT^*=M^*V$ where
   $V$ is given by equation (\ref{eq:mustV}).

  Moreover, for $f\in H^\infty$ and $h\in\mathcal H$,
 \begin{equation}
  \label{eq:Vstar}
    V^* (f\otimes h)= T_f R^*h.
 \end{equation}
\end{proposition}

\begin{proof}
   Suppose $V:\mathcal M\to H^2(k)\otimes \mathcal H$ is
   an isometry and $ V T_f^* = M_f^* V$ for all $f\in H^\infty(\mathbb A)$.
   Since $V:\mathcal M\to H^2(k)\otimes \mathcal H$,
   there exists operators $R_j:\mathcal M\to \mathcal H$ so that,
   for $h\in\mathcal M$,
 \begin{equation*}
   Vh=\sum \zeta_j \otimes R_j h,
 \end{equation*}
   with the sum converging SOT.  
  %
  Now, 
 \begin{equation*}
   \begin{split}
      \sum \zeta_j \otimes R_j T_m^* h = & VT_m^* h\\
       =& M_m^* Vh \\
       =& \sum M_m^* \zeta_j \otimes R_j h.
   \end{split}
 \end{equation*}
  Taking the inner product of both sides
  of the above equation with $1\otimes e$ ($e\in\mathcal H$) gives,
 \begin{equation*}
   \langle R_0 T_m^* h, e\rangle = \langle R_m h,e\rangle.
 \end{equation*}
   With $R=R_0$, this shows
   $R_m=R T_m^*$ and thus  proves that $V$ takes the form promised in
   equation \eqref{eq:mustV}. That this sum converges
   in norm follows from the spectral condition on $T$.

    To prove the conversely, the hypothesis that the sum converges WOT to the
    identity implies that $V$ defined as in equation (\ref{eq:mustV})
    (which converges in norm) is
    an isometry. We next prove equation \eqref{eq:Vstar},
   from which the conclusion that $M$ lifts
  $T$ via $ VT^* = M^* V$ will follow.

   To start, note that, for each $m\in\mathbb Z$,
 \begin{equation*}
   \begin{split}
      \langle V^* \zeta_m \otimes e, h \rangle
        =& \langle \zeta_m \otimes e, Vh \rangle \\
        =& \langle e, RT_m^* h \rangle \\
        =& \langle T_m R^* e, h\rangle.
  \end{split}
 \end{equation*}
   Hence $V^* \zeta_m \otimes e = T_m R^* e.$

   Next note that,
   from the computation above,
   equation \eqref{eq:Vstar} holds for
   Laurent polynomials (finite linear combinations
   of $\{\zeta_j : j\in\mathbb Z\}$).  Next,
   if $f\in H^2(k)$, then there
   is a sequence of  Laurent
   polynomials $p_n$ which converge
   to $f$ in $H^2(k)$ and also uniformly
   on compact subsets of $\mathbb A$. Hence,  $p_n\otimes h$
   converges in $H^2(k)\otimes \mathcal H$
   to $f\otimes h$ and also $T_{p_n}$ converges
   to $T_f$ in norm, and equation \eqref{eq:Vstar}
   is proved.

   Next, if both $f,g\in H^\infty$, then
 \begin{equation*}
   \begin{split}
    V^* M_f g\otimes e = & V^* fg  1\otimes e \\
     =& T_{fg} R^* e\\
     =& T_f T_gR^*e\\
     =& T_f V^* g\otimes e.
   \end{split}
 \end{equation*}
   Thus, $V^*M = TV^*$ so that $M$ lifts $T$.
\end{proof}

\begin{proposition}
 \label{prop:TinF}
  Suppose $T\in\mathcal B(\mathcal M)$ has spectrum
  in $\mathbb A$. If $\frac{1}{k}(T,T^*)\succeq 0$
  and if $R\in\mathcal B(\mathcal M,\mathcal H)$
  satisfies $R^*R=\frac{1}{k}(T,T^*)$, then
  then the sum in equation (\ref{eq:sumstoI})
  converges WOT to the identity.
  In particular, $M$ lifts $T$.

  Conversely, if $G$ is a positive operator and the sum
 \begin{equation*}
   \sum T_n G T_n^*
 \end{equation*}
  converges WOT to the identity, then $G=\frac{1}{k}(T,T^*)$.
\end{proposition}

\begin{remark}\rm
  It is always possible to choose $\mathcal H=\mathcal M$
  or $\mathcal H\subset \mathcal M$ ,
  though the former choice could lead to a 
  representation which is not minimal. 
\end{remark}

\begin{proof}
  The first part of the proposition follows from
 \begin{equation*}
   I= k(T,T^*)(\frac{1}{k}(T,T^*)(I))
     = k(T,T^*)(R^*R).
 \end{equation*}

   The hypothesis for the second part of the lemma
  is $k(T,T^*)(G)=I$. Hence,
 \begin{equation*}
    G = \frac{1}{k}(T,T^*)(k(T,T^*)(G)) = \frac{1}{k}(T,T^*)(I).
 \end{equation*}
\end{proof}

 Recall the notion of a minimal lifting given 
  in Definition \ref{def:minimal-lifting}.

\begin{proposition}
 \label{prop:minimal}
 The lifting $VT^*=M^*V$ of $T$ on $H^2(k)\otimes \mathcal H$ is minimal
 if and only if there does not exist a proper subspace
 $\mathcal F\subset \mathcal H$ such that
 the range of $V$ lies in $H^2(k)\otimes \mathcal F.$
\end{proposition}

\begin{proof}
  From the form of $V$, the smallest
  subspace  $\mathcal F$ of
  $\mathcal H$ such that the range of $V$ lies
  in $H^2(k)\otimes \mathcal F$
  is the closure of the range of $R=Q^*V$.
\end{proof}

\begin{proposition}
  Suppose $T\in\mathcal B(\mathcal M).$
  If $\sigma(T)\subset \mathbb A$ and
  $\frac{1}{k}(T,T^*)\succeq 0$,  then $M$ lifts $T$.

  If both $V_jT^*=M^*V_j$ where $M$ is acting on $H^2(k)\otimes \mathcal Hj$,
  $j=1,2$ are minimal liftings of $T$, then there is a unitary operator
  $U:\mathcal H_1\to \mathcal H_2$ so that $(I\otimes U)V_1=V_2$; i.e.,
  a minimal lifting is unique up to unitary equivalence.
\end{proposition}

\begin{proof}
  The first part follows from Proposition \ref{prop:TinF}.

  From Proposition \ref{prop:AEM},
 \begin{equation*}
    V_\ell h=\sum \zeta_j \otimes R_\ell T_j^* h,
 \end{equation*}
  where $R_\ell = Q_\ell^* V_\ell:\mathcal M\to \mathcal H_\ell$
  and $Q^*_\ell \sum \zeta_j\otimes f_j =f_0$
  on
  $H^2(k)\otimes \mathcal H_\ell$. Moreover,
 \begin{equation*}
   I =k(T,T^*)(R_\ell^* R_\ell)=   \sum T_j R_\ell^* R_\ell T_j^*.
 \end{equation*}
   Therefore, by Proposition \ref{prop:TinF}
  $R_\ell^* R_\ell =\frac{1}{k}(T,T^*)$ for $\ell =1,2.$

  From minimality, $R_\ell$ has dense range and therefore there is a unitary
  operator $U:\mathcal H_1\to \mathcal H_2$ so that $R_2= UR_1$.
  It follows that $(I\otimes U)V_1=V_2$.
\end{proof}

\section{Some Functional Hilbert Spaces}
 \label{sec:funhilby}
  Theorem \ref{thm:main}
  involves operator-valued measures and implicitly 
  certain related functional Hilbert spaces. In
  this section we sketch out the relevant constructions.
  Most of what is needed is summarized later
   as Lemma \ref{lem:omnibus} in Section \ref{sec:factor}.

\subsection{General constructions}
 \label{subsec:constructions}
  Let $\BT$ denote the Borel subsets 
  \index{$\BT$} of the unit circle $\mathbb T$.
  By an {\it operator-valued measure on $\mathbb T$} we mean
  a Hilbert space $\mathcal M$ and a function
\begin{equation*}
   \nu:\BT \to \mathcal B(\mathcal M)
\end{equation*}
  such that
 \begin{itemize}
   \item[(p)] $\nu(\omega)\succeq 0$ for $\omega \in \BT$; and
   \item[(ca)] for each $e,f\in\mathcal M$, the function
    \begin{equation*}
       \omega \mapsto \langle \nu(\omega)e,f\rangle
    \end{equation*}
       is a (complex) measure on $\BT$.
  \end{itemize}

   A (measurable) partition 
  \index{partition} $P$ of $\mathbb T$ is a finite disjoint collection
  $\omega_1,\dots,\omega_n \in \BT$ whose union is
  $\mathbb T$. A measurable simple function
  \index{simple function} $H$ is a function of the
  form
\[
  H=\sum_{j=1}^n K_{\omega_j} c_j
\]
  for some vectors $c_j\in\mathcal M$ and partition $P$.
   Here,  $K_{\omega}$  \index{$K_\omega$}
  denotes the characteristic function \index{chararcteristic function}
  of a set $\omega$.  Let $\mathcal S$ denote the collection
  of measurable simple functions. \index{$\mathcal S$}

  The measure $\nu$ gives rise to a semi-inner product on $\mathcal S$
  as follows. If $H^\prime=\sum_{\ell=1}^m  K_{\omega_\ell^\prime} c^\prime_\ell$
  is also in $\mathcal S$, define
\[
 \langle H^\prime,H \rangle_\nu 
   = \sum_{j,\ell} c_j^* \nu(\omega_j\cap\omega_\ell)c_\ell^\prime. 
\]
  In the usual way, this inner product gives rise to a semi-norm,
\[
  \|H\|_\nu^2 =\langle H,H\rangle_\nu.
\]

   A {\it tagging} \index{tagging} $S$ of the partition $P$ consists
  of a choice of  points $S=(s_j\in\omega_j)$.
  The pair $(P,S)$ is a {\it tagged partition.} \index{tagged partition}
  The collection of tagged partitions is 
  a directed set under the relation $(P,S)\preceq (Q,T)$
  if $Q$ is a refinement of $P$. 
  Given $F:\mathbb T\to \mathcal M$, let $F(P,S)$
  denote the resulting measurable simple function
 \begin{equation*}
     F(P,S) = \sum K_{\omega_j}F(s_j).
 \end{equation*}
  Thus, each such $F$ generates the net $\{F(P,S):(P,S)\}$ 
  of simple functions.

  Let $\mathcal R^2(\nu)$ denote those $F$ for which the
  net $\{F(P,S)\}$ is bounded and Cauchy in $\mathcal S$; i.e.,
  those $F$ for which there is a $C$ such that 
  $ \|F(P,S)\|_\nu \le C$ for all $(P,S)$, and such that for 
  each $\epsilon >0$ there is a partition $Q$ such
  that for any pair $(P,S),(P^\prime,S^\prime)$ 
  such that $P$ and $P^\prime$ both refine $Q$, 
\begin{equation}
 \label{eq:defRL}
 \epsilon^2 > \|F(P,S)-F(P^\prime,S^\prime)\|^2_\nu
   = \sum_{j,k} (F(s_j)-F(s_\ell^\prime))^* \nu(\omega_j\cap \omega_\ell^\prime)
    (F(s_j)-F(s_\ell^\prime)). 
\end{equation}

  The following are some  simple initial observation.

 \begin{lemma}
  \label{lem:simpleinR-two}
    Measurable simple functions are in $R^2(\nu)$.

    If $F\in R^2(\nu)$ and $H\in\mathcal S$, then
    the net $\langle H,F(P,S)\rangle_\nu$ is 
    Cauchy.

    If $F,G\in \mathcal R^2(\nu)$, then the net
    $\langle F(P,S),G(P,S)\rangle_\nu$ converges. 
\end{lemma}

\begin{proof}
  The first statement is evident.  

  Given tagged partitions $(P,S)$
  and $(Q,T)$, 
\[
 \begin{split}
   | \langle F(P,S),G(P,S)\rangle_\nu - & \langle F(Q,T),G(Q,T)\rangle_\nu |\\
   \le &  |\langle F(P,S)-F(Q,T),G(P,S)\rangle_\nu|
     + | \langle F(Q,T),G(P,S)-G(Q,T)\rangle_\nu |.
 \end{split}
\]
  This estimate, Cauchy-Schwarz, plus the boundedness
  hypothesis on the nets proves the third statement.

    The second statement is a special case of the third. 
\end{proof}

 \begin{lemma}
  \label{lem:sumsinR-two}
    If $F,G\in R^2(\nu)$, then so is $F+G$.
 \end{lemma}

 \begin{proof}
   The boundedness of the net $\{(F+G)(P,S)\}$ is evident. 
   Given tagged partitions $(P,S)$ and $(P^\prime,S^\prime)$,
   note that
  \[
   \begin{split}
     \|(F+G)(P,S)-&(F+G)(P^\prime,S^\prime)\|_\nu\\
       \le & \|F(P,S)-F(P^\prime,S^\prime)\|_\nu 
       +\|G(P,S)-G(P^\prime,S^\prime)\|_\nu.
  \end{split}
  \]
   Applying this estimate to appropriate partitions and common refinement
   proves the result.    
 \end{proof}

 The assignment,
\begin{equation*}
  \langle F,G\rangle_\nu = \lim \langle F(P,S),G(P,S)\rangle_\nu
\end{equation*}
 defines a semi-inner product on $\mathcal R^2(\nu)$
 which is also natural to write as
\begin{equation}
 \label{eq:defnuint}
  \langle F,G\rangle_\nu = \int \langle d\,\nu(s) F(s),G(s).
\end{equation}
 We define $L^2(\nu)$ as the completion, after moding
 out null vectors, of $\mathcal R^2(\nu)$ in the (semi-)norm
 induced by this (semi-)inner product.

\begin{proposition}
 \label{prop:simpledense}
   Simple functions are dense in $L^2(\nu)$.
   In particular, the inclusion $\mathcal M\to L^2(\nu)$
   which sends $m\in\mathcal M$ to the equivalence
   class of the constant function $m$ is bounded.

   Moreover, if $H=\sum_{1}^{n}  K_{\omega_j} m_j$ and
  $H^\prime= \sum_{1}^{n^\prime}  K_{\omega_j^\prime} m_j^\prime$, then
 \begin{equation*}
   \langle H,H^\prime \rangle_\nu
     =\sum_{j,\ell} \langle \nu(\omega_j\cap \omega_\ell^\prime)m_j,
          m_\ell^\prime \rangle.
 \end{equation*}
\end{proposition}

\begin{proof}
  Let $F\in\mathcal R^2(\nu)$ and $\epsilon >0$ be given.
  Choose a partition $Q$ such that for all 
  for all tagged partitions
  $(P,S),(P^\prime,S^\prime),$
  such that $P$ and $P^\prime$ refine $Q$,  
  the  inequality \eqref{eq:defRL} holds.
  Let $H=F(Q,T)$. Then, 
\[
  \epsilon^2 > \|(H-F)(P,S)\|^2_\nu 
    = \langle H,H\rangle_\nu -\langle H,F(P,S)\rangle_\nu 
  - \langle F(P,S),H\rangle_\nu
      + \langle F(P,S),F(P,S)\rangle_\nu.
\] 
  In view of Lemma \ref{lem:simpleinR-two},
  the right hand side converges to $\|H-F\|^2_\nu$ and so (measurable)
  simple functions are dense in $\mathcal R^2(\nu)$.  Since
  $\mathcal R^2(\nu)$ is dense in $L^2(\nu)$ the first statement
  follows. 

  The second statement is a restatement of the definition of the
  inner product induced by $\nu$ on measurable simple functions.
\end{proof}

 While there is no reason to believe a given continuous $\mathcal M$
 valued function on $\mathbb T$ should be in $L^2(\nu)$, there is
 an important class which is.  

\begin{proposition}
 \label{prop:L2-special}
   Suppose $f:\mathbb T\to \mathcal B(\mathcal M)$ is continuous
   and $C$ is a non-negative real number. 
   If, for each $s$ and $t$  and Borel set $\omega$,
   both 
\begin{equation}
 \label{eq:L2-special-bound}
   f(s)\nu(\omega)f(s)^* \le C \nu(\omega)
\end{equation}
 and  
 \begin{equation}
  \label{eq:L2-special}
     (f(s)-f(t)) \nu(\omega) (f(s)-f(t))^*  \preceq \|f(s)-f(t)\|^2
          \nu(\omega),
 \end{equation}
  then for each $m\in\mathcal M$, the function $f(s)m$ is in $L^2(\nu)$. 
\end{proposition}

\begin{proof}
  Fix a vector $m$ and let $F(s)=f(s)^*m$. 
  The inequality of equation
  \eqref{eq:L2-special-bound} implies 
  the net $\{F(P,S)\}$ 
  is bounded.  A straightforward argument using the uniform continuity
  of $f$ and the inequality \eqref{eq:L2-special}
  shows that the net $\{F(P,S)\}$ is Cauchy. 
  Hence $F\in\mathcal R^2(\nu)$. 
\end{proof}

The algebra $C(\mathbb T)$ of continuous (scalar-valued) functions
 on $\mathbb T$ has a natural representation on $L^2(\nu)$. 

\begin{lemma}
 \label{lem:tau-rep}
   If $a\in C(\mathbb T)$ and $F\in \mathcal R^2(\nu)$,
   then $aF\in\mathcal R^2(\nu)$ and moreover,
   $\|aF\|_\nu \le \|a\|_\infty \|F\|_\nu$.  Hence
   $a$ determines a bounded linear operator $\tau(a)$
  on $L^2(\nu)$.  The mapping sending $a\in C(\mathbb T)$
  to $\tau(a)\in \mathcal B(L^2(\nu))$ is a unital $*$-representation. 

  Finally, given $a,a^\prime\in C(\mathbb T)$ and 
  simple measurable functions
   $F=\sum K_{\omega_j} m_j$ and 
  $F^\prime = \sum K_{\omega_\ell^\prime}m_\ell^\prime$,
\begin{equation}
 \label{eq:tau-simple}
  \langle a^\prime F^\prime, aF\rangle_\nu
   = \sum_{j,\ell} \int_{\omega_j\cap \omega_\ell^\prime}
           a^\prime(s) a^*(s) \langle d\nu(s)m_\ell^\prime,m_j\rangle.
\end{equation}
\end{lemma}

\begin{proof}
   Fix $F\in\mathcal R^2(\nu)$. 
    For any partition
   $P=(\omega_j)$ of $\mathbb T$ and pointing $S=(s_j\in\omega_j)$,
 \begin{equation*}
  \begin{split}
   \|(aF)(P,S)\|_\nu^2 
     = & 
   \sum \langle   \nu(\omega_j) a(s_j) F(s_j),
        a(s_j) F(s_j) \rangle  \\
     = & \sum |a(s_j)|^2 \langle \nu(\omega_j\cap\omega) F(s_j),F(s_j)\rangle \\
     \le & \|a\|_\infty^2 \|F(P,S)\|_\nu^2.
  \end{split}
 \end{equation*}
   Thus, since the net $\{F(P,S)\}$
   is bounded, so is the net $\{aF(P,S)\}.$ 

   If $(R,T)$ is another tagged partition,
   where $R=(\theta_\ell)$ and $T=(t_\ell \in \theta_\ell),$ then
  $ (aF)(P,S)-(aF)(R,T) = G+H$, where
\[
 \begin{split}
   G=& \sum_{j,\ell} (a(s_j)-a(t_\ell))K_{\omega_j\cap\theta_\ell} F(s_j), \\
   H=& \sum_{j,\ell} a(t_\ell)K_{\omega_j\cap\theta_\ell}(F(s_j)-F(t_\ell)).
 \end{split}
\]
  If $\epsilon$ bounds both $|a(s_j)-a(t_\ell)|$ and $\|F(P,S)-F(R,T)\|_\nu$
  and if $C$ is a bound for the net $\{F(P,S)\}$, then 
\[
   \|G\|_\nu \le \epsilon C, \ \ \|H\|_\nu \le \|a\|_\infty \epsilon. 
\]
  Thus, using the
  uniform continuity
  of $a$ and the fact that the net $\{F(P,S)\}$
  is Cauchy, it is possible to choose 
  a partition $Q$ of sufficiently small width so that
  if 
  $P$ and $R$ are refinements of $Q$ with taggings $S$ and $T$ 
  respectively, then 
\[
 \begin{split} \|(aF)(P,S)-(aF)(R,T)\|_\nu
   = & \|G+H\|_\nu \\
  \le & \|G\|_\nu +\|H\|_\nu \le (C+\|a\|_\infty) \epsilon.
 \end{split}
\]
 Thus the net $\{(aF)(P,S)\}$ is Cauchy.  Hence  $aF\in \mathcal R^2(\nu)$. 

   It suffices to prove equation \eqref{eq:tau-simple}
   in the case that $F=K_{\omega}m$ and $F^\prime = K_{\omega^\prime}m^\prime.$
   Given a partition $(P,S)$, 
 \begin{equation*}
  \begin{split}
      \langle (a^\prime F^\prime)(P,S),(aF)(P,S)\rangle_\nu 
       =& \sum a^\prime(s_j)^* a(s_j) 
         \langle\nu(\omega_j\cap\omega\cap\omega^\prime)m^\prime
          m \rangle \\
       =& \int_{\omega\cap\omega^\prime} \sum a^\prime(s_j)^* a(s_j) 
             K_{\omega_j} \langle d\, \nu(s)m^\prime,m\rangle. 
    \end{split}
 \end{equation*}
   Given $\epsilon>0,$ if 
   the partition $P$ is chosen, using the uniform continuity
   of $a^\prime a^*$, so that 
\[
 \| [a^\prime a^* - \sum_j a^\prime(s_j)a^*(s_j)K_{\omega_j} ] 
  K_{\omega\cap \omega^\prime}\|_\infty < \epsilon,
\]
   then
\[
 | \int_{\omega\cap\omega^\prime} [a^\prime a^*- \sum (a^\prime(s_j))^* a(s_j) 
             K_{\omega_j}] \langle d\, \nu(s)m^\prime,m\rangle |
   \le \epsilon \|\nu(\omega\cap\omega^\prime)\| \, \|m^\prime\| \, 
    \|m\|.  
\]
  It follows that the net $\langle (a^\prime F^\prime)(P,S),(aF)(P,S)\rangle_\nu$
  converges to the integral  
\[
  \int_{\omega\cap\omega^\prime} a^\prime (s) a^*(s) \langle d\, 
     \nu(s)m^\prime,m\rangle,
\]
  completing the proof of equation \eqref{eq:tau-simple}.

   Each $a$ determines a bounded operator on $R^2(\nu)$ (with
   norm at most $\|a\|_\infty$) and hence extends to
   a bounded operator $\tau(a)$ on all of $L^2(\nu)$. 
   It remains to prove that $\tau$ determines a unital $*$-representation on
   $L^2(\nu)$.  
   Evidently $\tau(1)=I.$ Using equation \eqref{eq:tau-simple}
   twice (first with $a=1$ and the second with $a=(a^\prime)^*$
  and $a^\prime =1$),
 \begin{equation*}
  \begin{split}
   \langle \tau(a^\prime)^* F,F^\prime\rangle_\nu
      = & \langle F, \tau(a^\prime)F\rangle_\nu \\
     = & \langle F, a^\prime F^\prime\rangle_\nu \\
     = & \int (a^\prime)^*(s) \langle d\mu(s)m,m^\prime\rangle \\
     = & \langle \tau((a^\prime)^*) F,F^\prime \rangle_\nu.
  \end{split}
 \end{equation*}
   Hence $\tau(a)^* =\tau(a^*)$.

   Finally, again using equation \eqref{eq:tau-simple} twice,
   this time first with $a=a a^\prime$ $a^\prime = 1$,
   and second with $a=a$ and $a^\prime = (a^\prime)^*$,
 \begin{equation*}
  \begin{split}
    \langle \tau(a a^\prime) F,F^\prime \rangle_\nu
     = & \int a^\prime a(s) \langle d\mu(s)m,m^\prime \rangle_\nu \\
     = & \int ((a^\prime)^*)^* a(s) \langle d\mu(s)m,m^\prime \rangle \\
     = & \langle \tau(a^\prime)^* \tau(a) F,F^\prime \rangle_\nu \\
     = &  \langle \tau(a^\prime)\tau(a) F,F^\prime \rangle_\nu.
  \end{split}
 \end{equation*}
   Thus $\tau(a^\prime a)=\tau(a^\prime)\tau(a)$.
\end{proof}

\subsection{Agler decompositions again}
  Suppose $\mu$ is an Agler decomposition as
  defined in subsection \ref{subsec:Agler-SoS}.
  Then both $\mu,$ and $\Lambda$ defined by
 \begin{equation*}
   \Lambda(\omega)= k(T,T^*)(\mu(\omega)),
 \end{equation*}
  are positive $\mathcal B(\mathcal M)$-valued
  measures on $\BT$ and the constructions of
  the previous section
  apply to both $L^2(\mu)$ and $L^2(\Lambda)$.

\begin{lemma}
 \label{lem:Lambda-mu}
  If $F\in \mathcal R^2(\Lambda)$, then $F\in \mathcal R^2(\mu)$ and
  $\langle F,F\rangle_\Lambda \ge \langle F,F\rangle_\mu$.
  Thus, the mapping $\Phi^*:\mathcal R^2(\Lambda)\to \mathcal R^2(\mu)$
  given by $F\mapsto F$ induces a contractive linear mapping
  $\Phi^*:L^2(\Lambda)\to L^2(\mu)$. \index{$\Phi$}
\end{lemma}

\begin{proof}
  This follows immediately from
 \begin{equation*}
   \Lambda(\omega)= k(T,T^*)(\mu(\omega)) \succeq\mu(\omega).
 \end{equation*}
\end{proof}

  Given $m\in\mathcal M$, let $Y$ denote the mapping
  $Y:\mathcal M \to L^2(\Lambda)$ \index{$Y$}
  defined by $Ym(\psi)=T_{\psi}^*m$. 
  Here the identification of $\Psi$, the collection
  of test functions, with $\mathbb T$ is in force. 
  Of course, it needs to be verified that $Ym(\psi)$
  is indeed in $L^2(\Lambda)$. 
  Let $\iota$ \index{$\iota$}
  denote the inclusion, as constant functions,
  of $\mathcal M$ into $R^2(\Lambda)$.
  Thus, if $m\in\mathcal M$,  then $\iota m$ denotes the constant
  function $\iota m(\psi)=m$.

\begin{lemma}
 \label{lem:Y-bounded}
  For $m\in\mathcal M$, the function $Ym$ is in $\mathcal R^2(\Lambda)$.
 
  Moreover,
 \begin{equation*}
   \langle \Lambda(\mathbb T)m,m\rangle_{\mathcal M} =
    \langle \iota m,\iota m\rangle_{L^2(\Lambda)}
      \ge \langle Ym,Ym\rangle_{L^2(\Lambda)}.
 \end{equation*}
  Thus, $Y$ determines a bounded linear operator 
  $Y:\mathcal M\to L^2(\Lambda)$ given
  by $(Ym)(\psi) = T_\psi^* m.$  In the notation of
  equation \eqref{eq:defnuint},
\begin{equation}
 \label{eq:YstarY}
  \langle Y^*Ym,m\rangle_{\Lambda} = 
  \langle \int d\, \Lambda(\psi) T_\psi^* m, T_{\psi}^* m \rangle.
\end{equation}

  Further, $\Phi^* Y:\mathcal M\to L^2(\mu)$
  is bounded and 
 \begin{equation}
  \label{eq:YPhiPhiY}
   \langle Y^* \Phi \Phi^* Ym,m^\prime\rangle_{\nu}
    = \int \langle d\, \mu(\psi) T_\psi^* m,T_{\psi}^* m^\prime \rangle.
 \end{equation}
\end{lemma}

\begin{remark}\rm
 \label{rem:YstarY}
   We interpret equations \eqref{eq:YstarY} and 
   \eqref{eq:YPhiPhiY} as
\begin{equation}
 \label{eq:YstarYprime}
   Y^*Y = \int T_{\psi} d\,\Lambda(\psi) T_{\psi}^*
\end{equation}
  and
\begin{equation}
 \label{eq:YPhiPhiYprime}
  Y^* \Phi \Phi^* Y = \int T_{\psi} d\,\mu(\psi) T_{\psi}^*
\end{equation}
 respectively.

 Given a tagged partition $(P,S)$, let
\[
 \Delta(P,S,\Lambda) = \sum T_{s_j} \Lambda(\omega_j)T_{s_j}^*
\]
 and define $\Delta(P,S,\mu)$ similarly.
  Thus, $\Delta(P,S,\Lambda)$ is an operator on $\mathcal M$
 and because  $T_s \Lambda(\omega) T_s^* \le \Lambda(\omega)$,
  it is positive semidefinite and bounded above by $\Lambda(\mathbb T)$. 
  For vectors $m,m^\prime\in\mathcal M$,
\[
  \langle \Delta(P,S,\Lambda)m,m^\prime \rangle = 
   \langle Ym(P,S),Ym^\prime(P,S)\rangle_\Lambda. 
\]
 Thus, the net $\{\Delta(P,S,\Lambda)\}$ converges WOT to the operator
 of equation \eqref{eq:YstarYprime}. 

 It follows that the net $\{\frac{1}{k}(T,T^*)(\Delta(P,S,\Lambda))\}$
 also converges. On the other hand,
\[
 \frac{1}{k}(T,T^*)(\Delta(P,S,\Lambda)) = \Delta(P,S,\mu).
\]
 Hence the net $\{\Delta(P,S,\mu)\}$ converges WOT to the 
  operator of equation \eqref{eq:YPhiPhiYprime}.
\end{remark}

\begin{proof}
  By hypothesis, for $\varphi$ in the scalar Schur class 
  and measurable sets $\omega$, 
\begin{equation}
 \label{eq:L2-Tpsi}
  \Lambda(\omega) \succeq T_\varphi \Lambda(\omega) T_\varphi^*.
\end{equation}
  Thus, the functions $Ym$ 
 satisfies the hypotheses, with respect to $\Lambda$, of 
  Proposition \ref{prop:L2-special}.  It follows
  that $Ym$ is in $L^2(\Lambda)$ for each $m$.

  The moreover follows immediately from equation \eqref{eq:L2-Tpsi}.  

  The rest of the Lemma follows from the definitions. 
\end{proof}

\begin{lemma}
 \label{lem:lurk}
  Let $\mu$ be an Agler decomposition of
  the pair $(T,X)$ and let, as in Proposition \ref{prop:AEM},
  $R^*R=\frac{1}{k}(T,T^*)$. Then,
\begin{equation*}
    R^*R + Y^*\Phi \Phi^* Y
   = XR^*RX^* + \iota^* \Phi \Phi^* \iota.
 \end{equation*}
\end{lemma}

\begin{proof}
 Part (ii) of the definition of an Agler decomposition can be written as
 \begin{equation*}
   \frac{1}{k}(T,T^*)(I) + \int_\Psi T_\psi d\mu(\psi)T_\psi^*
     = \frac{1}{k}(T,T^*)(XX^*) + \mu(\mathbb T).
 \end{equation*}
   Because $X$ commutes with $T^*$,
 \begin{equation*}
   \frac{1}{k}(T,T^*)(XX^*) = X \frac{1}{k}(T,T^*)(I)X^*
 \end{equation*}
  and hence $\frac{1}{k}(T,T^*)(XX^*)= X R^*RX^*.$
   An application of the last part of Lemma \ref{lem:Y-bounded}
   gives
 \[
   R^*R + Y^*\Phi\Phi^* Y = X R^*R X^* + \mu(\mathbb T).
 \]
  
  Noting that 
 \[
   \langle \iota^* \Phi \Phi^* \iota m,m^\prime\rangle
     = \langle m, m^\prime\rangle_{L^2(\mu)} = \langle \mu(\mathbb T)m,m^\prime
   \rangle
 \]
  completes the proof.    
\end{proof}

\section{Uniformity of the Test Functions}
 \label{sec:uniform}
  Using the orthonormal basis $\{\zeta_j\}$ for $H^2(k)$ 
  defined in equation \eqref{eq:zeta-basis}, each test
  function $\psi$ has a Laurent expansion,
\[
  \psi = \sum \langle \psi,\zeta_j\rangle \zeta_j.
\]
 In this section we show that 
\[
  T_{\psi}^* = \sum \langle \zeta_j,\psi\rangle T_j^*
\]
  with convergence in the strong operator topology.

 The section begins with establishing a uniform, 
 independent of $\psi$,  estimate
 on the rate of convergence of the Laurent series
 for $\psi$ on compact subsets of $\mathbb A$.

\begin{lemma}
 \label{lem:C-rho}
  There is a $0<\rho<1$ and a constant
  $C$ so that for all $\psi\in\Psi$
  and $j\in\mathbb Z$,
 \begin{equation*}
  |\langle \psi,\zeta_j\rangle |< C \rho^{|j|}.
 \end{equation*}
\end{lemma}

\begin{proof}[Sktech of proof]
  There is a function $\varphi$
  analytic in a neighborhood of our
  annulus $\mathbb A$ such that
 \begin{itemize}
  \item[(a)] for $|z|=1$, $|\varphi(z)|=1$;
  \item[(b)] for $|z|=q$, $|\varphi(z)|=\sqrt{q}$; and
  \item[(c)]  $\varphi(\sqrt{q})=0.$
 \end{itemize}
  It extends by reflection across both
  boundaries to be analytic in the
  annulus $\{ q^{\frac32}<|z|<q^{-\frac12}\}$
  (see Section \ref{sec:more-test-functions}).

  It follows that, up to a unimodular
  constant, if $\psi$ is unimodular
  on the boundary of $\mathbb A$
  and has exactly two zeros,
  these being at $\sqrt{q}$ and
  $\sqrt{q}\gamma$ (for a necessarily
  unimodular $\gamma$), then
 \begin{equation}
  \label{eq:parametrize}
   \psi(z) = \delta \frac{\varphi(z)\varphi(\gamma^* z)}{z},
 \end{equation}
  for some unimodular $\delta$.  In particular
  equation \eqref{eq:parametrize} gives
  an explicit parametrization of $\Psi$ by $\mathbb T$.

  It now follows that $\psi\in\Psi$ is bounded uniformly
  (independent of $\psi$)
  on a larger annulus than $\mathbb A$ and
  the result follows.
\end{proof}

  In the following Lemma $\mu$ is an Agler decomposition
  for $(T,X)$. Thus, $\Lambda(\omega)=k(T,T^*)(\mu(\omega))$
  and  for $\varphi$ in the
  scalar Schur class,
  $T_\varphi \Lambda(\omega)T_\varphi^* \preceq \Lambda(\omega)$.

\begin{lemma}
 \label{lem:intfour}
   If $m\in\mathcal M$, then, for each $j$, the
   function $\langle \zeta_j,\psi\rangle T_j^*m \in L^2(\Lambda)$
   and moreover, independent of $j$, there is a $C>0$ and $0<\rho<1$ such
   that
 \begin{equation*}
   \| \langle \zeta_j,\psi\rangle T_j^*m \|_{L^2(\Lambda)}\le C\rho^{|j|}.
 \end{equation*}

   If $G\in L^2(\Lambda)$ and $m\in\mathcal M$, then
  \begin{equation}
   \label{eq:intfour}
    \langle G, Ym \rangle_{L^2(\Lambda)}
      = \sum_j \langle G, \langle \zeta_j,\psi\rangle T_j^*m
         \rangle_{L^2(\Lambda)}.
  \end{equation}

   If $F$ is a measurable simple function,  then
  \begin{equation*}
     \langle F, \Phi^* Ym\rangle_{L^2(\mu)} =
      \sum_j \langle F, \langle \zeta_j, \psi\rangle T_j^* m\rangle_{L^2(\mu)}.
  \end{equation*}
\end{lemma}

\begin{proof}
  Given a positive integer $N$, define 
  $\sigma_N:\mathbb T\to H^\infty(\mathbb A)$
  by 
 \[
   \sigma_N(\psi) = \sum_{|j|\le N} \langle \zeta_j,\psi\rangle \zeta_j.
 \]
   In view of Lemma \ref{lem:C-rho}, the sequence $\sigma_N$ converges to 
   the identity function $\psi$ uniformly
   on compact subsets of $\mathbb A$. Hence,  
   by Proposition \ref{prop:L2-special}, for each $m\in\mathcal M$
 \[
     \| T_{\psi-\sigma_N} ^* m \|_{L^2(\Lambda)}
      = \|T_\psi^*m 
       - \sum_{|j|\le N} \langle \zeta_j,\psi\rangle T_j^* m\|_{L^2(\Lambda)}
 \]
  converges to $0$ and equation \eqref{eq:intfour} follows. 

 To finish the proof, choose $G=\Phi F$ in
  equation \eqref{eq:intfour}
  to obtain
 \begin{equation*}
  \begin{split}
   \langle \Phi F, Ym \rangle_{L^2(\Lambda)}
      = & \sum_j \langle \Phi F, \langle \zeta_j,\psi\rangle T_j^*m
       \rangle_{L^2(\Lambda)} \\
     = & \sum_j \langle F, \langle \zeta_j,\psi\rangle T_j^*m
       \rangle_{L^2(\mu)}.
  \end{split}
 \end{equation*}
 using that $\Phi^*:L^2(\Lambda)\rightarrow L^2(\mu)$ is the 
  inclusion mapping (and is bounded). 
  The final conclusion of the lemma follows.
\end{proof}

\section{The Factorization and Lurking Isometry}
 \label{sec:factor}
  The next several sections, Sections \ref{sec:factor}, \ref{sec:colligation},
  and \ref{sec:proof},  are devoted to the proof of
  (ad) implies (sc) in Theorem \ref{thm:main} and throughout
  these sections the relevant hypotheses are in force.
  Namely, $\mathcal M$ is a separable Hilbert space,
 \begin{enumerate}[(a)]
  \item  $X,T\in\mathcal B(\mathcal M)$ commute;
  \item  $\sigma(T)\subset \mathbb A$;
  \item $T$ lifts to $M$ on $H^2(k)\otimes \mathcal M$ via
    $ VT^* =M^*V$ and
  \begin{equation*}
      Vh = \sum \zeta_j \otimes RT_j^* h,
  \end{equation*}
    where $R^*R = \frac{1}{k}(T,T^*)$; and
  \item 
   \label{it:mu}
    there exists a measure $\mu:\BT\to \mathcal B(\mathcal M)$
    such that, with $\Lambda(\omega)=k(T,T^*)(\mu(\omega)),$ 
  \begin{equation*}
    \Lambda(\omega)) -
             T_\varphi \Lambda(\omega) T_\varphi^* \succeq 0
  \end{equation*}
    for all Borel subset $\omega$ and Schur class functions $\varphi$ and
  \begin{equation*}
    \frac{1}{k}(T,T^*) (I-XX^*) = \mu(\mathbb T)
       -\int T_\psi d\mu(\psi)T_\psi^*.
  \end{equation*}
 \end{enumerate}

  Once properly formulated to account for infinitely many test functions, 
  the overarching strategy for proving results like
  Theorem \ref{thm:main} is 
  now well established, but the presence of infinitely many, and not necessarily
  orthogonal, test functions requires some reinterpretation
  of earlier results, revealing new structures.  
  The positivity condition in (ad) (item \eqref{it:mu} above) 
  is {\it factored} and this
  factorization produces a {\it lurking isometry} and of course
  an auxiliary Hilbert space.  The lurking
  isometry in turn generates the $\Psi$-unitary colligation.
  A good deal of effort is required to show that the
  resulting transfer function solves the problem and the
  argument given here is patterned after that in \cite{MS},
   which in turn borrowed from 
  \cite{BLTT} \cite{Archer} and closely related to those
  in \cite{AE}. 

  The factorization we will need comes from {\it factoring} the
  measure $\Lambda$ of Remark \ref{rem:Lambda}.
  This factorization amounts to the
  construction of the Hilbert spaces 
  $L^2(\Lambda)$ and $L^2(\mu)$ in 
  Section \ref{sec:funhilby}. 
  The following Lemma summarizes many of the 
  needed results and constructions
  from  Section \ref{sec:funhilby} 

\begin{lemma}
 \label{lem:omnibus}
  With the hypotheses above,
 \begin{enumerate}
  \item[(i)] there exist  Hilbert spaces
   $L^2(\Lambda)$ and $L^2(\mu)$ which contain densely
   all simple measurable $\mathcal M$-valued functions so that,
   in particular, the inclusion mapping $\iota:\mathcal M\to L^2(\Lambda)$
   \index{$\iota$ inclusion map} is bounded (not necessarily isometric);
  \item[(ii)]  the space  $L^2(\Lambda)$ includes in $L^2(\mu)$
    contractively so that there exists an  operator $\Phi$ 
    \index{$\Phi$} whose adjoint
   $\Phi^*: L^2(\Lambda) \to L^2(\mu)$ is the inclusion mapping; and
  \item[(iii)] an operator $Y:\mathcal M\to L^2(\Lambda)$ defined by
   $Ym=T_\psi^* m$ (that is, the function $Ym(\psi)=T_\psi^* m$ 
    \index{$Y$} 
   determines an element of $L^2(\Lambda)$),
 \end{enumerate}
    which together satisfy the lurking isometry \index{lurking isometry}
    equality,
 \begin{equation}
  \label{eq:lurking}
    R^*R + Y^*\Phi \Phi^* Y
   = XR^*RX^* + \iota^* \Phi \Phi^* \iota.
 \end{equation}

   Moreover, if 
 \begin{enumerate}[(a)]
   \item $a,a^\prime:\mathbb T\to \mathbb C$ are continuous;
   \item $\omega,\omega^\prime$ are Borel subsets of $\mathbb T$; 
   \item $m,m^\prime \in\mathcal M$; and
   \item $F = K_\omega m$ and $F^\prime =K_{\omega^\prime}m^\prime$,
 \end{enumerate}
   then $aF$ and $a^\prime F^\prime$ are in $L^2(\mu)$ and
 \begin{equation*}
    \langle aF,a^\prime F^\prime \rangle
       = \int_{\omega\cap\omega^\prime}
          a(\psi) a^\prime(\psi)^* \langle d\mu(\psi)m,m^\prime\rangle.
 \end{equation*}
   In particular, if
   $F\in L^2(\mu)$ is simple, then $aF$ determines
   an element of $L^2(\mu)$ and $\|aF\|\le \|a\|_\infty \|F\|$.
  Thus,
  there is a unital $*$-representation
  $\tau:C(\mathbb T)\to \mathcal B(L^2(\mu))$  such
  that $\tau(E(z)) F (\psi)$ $= \psi(z) F(\psi)$. \index{$\tau$}
  (Recall the identification of $\Psi$, the collection of test
  functions, with $\mathbb T$.)
\end{lemma}

  Condition (i) in the definition of Agler decomposition
  implies $Y$ is a bounded (in fact contractive) operator
  into $L^2(\Lambda)$. (Details in Section \ref{sec:funhilby}).
  Further, $\Phi^* Ym = T_\psi^* m$ determines
  an element of $L^2(\mu)$ and in  condition (ii) in the
  definition of an Agler decomposition equation
  \eqref{eq:key-mu} becomes,
 \begin{equation*}
   \langle \Phi^* Ym,\Phi^* Ym\rangle
    =\int \langle T_\psi d\mu(\psi) T_\psi^* m,m \rangle.
 \end{equation*}
  Thus
 \begin{equation*}
     \frac{1}{k}(T,T^*)- X\frac{1}{k}(T,T^*)X^*
        = \iota^* \Phi \Phi^*\iota -  Y^* \Phi \Phi^* Y.
 \end{equation*}
  Rearranging and using the relation
  $\frac{1}{k}(T,T^*)=R^*R$ of Proposition \ref{prop:TinF}
  produces the lurking isometry equality of equation \eqref{eq:lurking}.

\section{The colligation and its Transfer Function}
 \label{sec:colligation}
  Recall there are two parts to the colligation.
  The unitary matrix and the representation.

\def\cW{\mathcal W}

\subsection{The unitary matrix}
  The lurking isometry,  equation \eqref{eq:lurking}, produces, non-uniquely,
  the unitary matrix of item (iii) of Definition \ref{def:colligation}.  
  The construction requires
  an initial enlargement of the space $L^2(\mu)$.
  Let $\ell^2$ denote the usual separable Hilbert space
  with orthonormal basis $\{e_j : j\in\mathbb N\}$
  and define $\cW:L^2(\mu) \to L^2(\mu)\otimes \ell^2$
  by $\cW F= F\otimes e_0$. In particular, $\cW$ is an isometry.

  Let $\mathcal K$ and $\mathcal K_*$ denote the subspaces of
   $[L^2(\mu)\otimes \ell^2] \oplus \mathcal H$
  given by the closures of the spans of
 \begin{equation*}
   \{ \begin{pmatrix} \cW\Phi^* Ym \oplus  Rm \end{pmatrix} :
             m\in \mathcal M\}, \ \ \
   \{ \begin{pmatrix} \cW\Phi^* \iota m \oplus  RX^*m \end{pmatrix} :
              m\in \mathcal M\}
 \end{equation*}
   respectively, where $\iota$ \index{$\iota$} is the inclusion of
   $\mathcal M$ into $L^2(\Lambda)$.
    The lurking isometry of equation (\ref{eq:lurking}) says that the
  mapping from  $\mathcal K$ to $\mathcal K_*$ defined by
 \begin{equation*}
     \begin{pmatrix} \cW\Phi^* Ym \oplus  Rm \end{pmatrix}
            \to
   \begin{pmatrix} \cW\Phi^* \iota m \oplus  RX^*m \end{pmatrix}
 \end{equation*}
  is an isometry. Because $\mathcal K$ and $\mathcal K_*$ have the same codimension
  (i.e., their orthogonal complements have the same dimension),
  this isometry can be extended to a unitary 
 \begin{equation*}
   U=\begin{pmatrix} A^* & C^*\\ B^* & D^*\end{pmatrix}:
       \begin{matrix} L^2(\mu)\otimes \ell^2 \\ \oplus \\ \mathcal H \end{matrix} \to
       \begin{matrix} L^2(\mu)\otimes \ell^2 \\ \oplus \\ \mathcal H \end{matrix}, 
 \end{equation*}
  giving  rise to the usual system of equations,
 \begin{equation}
  \label{eq:system}
   \begin{split}
     A^* \cW\Phi^* Y  + C^* R =&  \cW\Phi^* \iota\\
     B^* \cW\Phi^* Y + D^*R =& RX^*.
   \end{split}
 \end{equation}
  Note that the domain of $D$ and $B$ and the codomain
  of $C$ is $\mathcal M.$


\subsection{The representation}
   Of course we also need the representation
   $\rho:C(\mathbb T)\to \mathcal B(L^2(\mu)\otimes \ell^2)$
   of item (ii) in Definition \ref{def:colligation}. 
   We begin with the unital representation
   $\tau:C(\mathbb T)\to \mathcal B(L^2(\mu))$
   from Lemma \ref{lem:omnibus} (see also Lemma \ref{lem:tau-rep}) and define
   $\rho=\tau\otimes I$, where $I$ is the identity on $\ell^2$.
   \index{$\rho$}

\subsection{The transfer function and its properties}
   Let $E(z):\Psi\to\mathbb C$ denote evaluation at $z;$ \index{$E(z)$}
   i.e., $E(z)(\psi)=\psi(z).$ For $F\in L^2(\mu)$,
   $\tau(E(z)) F (\psi)= \psi(z)F(\psi)$.
   The corresponding transfer function is then given by
 \begin{equation}
   \label{eq:the-transfer}
    W(z) = D+C(I-\rho(E(z))A)^{-1} \rho(E(z))B.
 \end{equation}
    The function $W$ gives rise to the multiplication operator $M_W$
    on $H^2(k)\otimes \mathcal M$. In the following subsection
    we make some observations
    related to $M_W$ and the corresponding
    $\Psi$-unitary colligation needed in the sequel.

   There is a canonical auxiliary multiplication operator associated
   to $z\mapsto \rho(E(z))$ which, as in equation \eqref{eq:mult-op},
   is most conveniently defined in terms of its adjoint.
   \index{$Z$}
   Define $Z^*:H^2(k)\otimes [ L^2(\mu)\otimes \ell^2]
      \to H^2(k)\otimes [L^2(\mu)\otimes \ell^2]$ by
 \begin{equation*}
    Z^* (k_z\otimes F )= k_z \otimes \rho(E(z))^*F.
 \end{equation*}
   Of course it needs to be checked that, after
   extending by linearity,  this prescription
   produces a bounded operator, a fact that follows  readily from
 \begin{equation*}
  \begin{split}
    \langle k_z \otimes \sum F_j\otimes e_j ,& k_w \otimes \sum G_\ell \otimes e_\ell \rangle
      - \sum_j \langle Z^* k_z \otimes F_j, Z^* k_w \otimes G_j  \rangle \\
     = & \sum_j k(w,z) [ \langle F,G\rangle - \langle \rho(E(z))^*F_j,\rho(E(w))^*G_j\rangle \\
     =&\sum_j \int k(w,z) (1-\psi(z)^*\psi(w))G(\psi)^*d\mu (\psi)F(\psi)
  \end{split}
 \end{equation*}
    and the fact that each $k(z,w)(1-\psi(z)\psi(w)^*)$ is a positive kernel
    and $\mu$ is a positive measure. Here we have used
    Proposition \ref{prop:testandkernel} and have actually
   proved that $Z$ has norm at most one.

   Thus $\rho(E(z))$ determines a (multiplication) operator
   on $H^2(k)\otimes [L^2(\mu)\otimes \ell^2]$ denoted 
   by $Z$:
 \begin{equation*}
  \langle Z{\tt{f}}, k_z\otimes F\rangle_{H^2(k)\otimes [L^2(\mu)\otimes \ell^2]}
     = \langle \rho(E(z)){\tt f}(z),F\rangle_{L^2(\mu)\otimes \ell^2}.
 \end{equation*}

 \begin{lemma}
  \label{lem:Z1F}
    Given a simple measurable function 
    $F=\sum K_{\omega_\ell}m_\ell \in L^2(\mu)$ and $f\in H^\infty$, 
  \begin{equation}
   \label{eq:Z1F}
    Z(f\otimes (F\otimes e_p)) =
      \sum f \zeta_j \otimes \langle \psi, \zeta_j  \rangle F \otimes e_p.
  \end{equation}
   Here $K_{\omega_\ell}$ is the characteristic function of the
   Borel set $\omega\subset \mathbb T;$ $e_p$ is the element
   of $\ell^2$ with a $1$ in the $p$-th entry and $0$ elsewhere;
   and the symbol $\psi$ denotes the variable in $\Psi$.
 \end{lemma}

  In particular, the sum on the right hand side converges.
  Since $\langle \psi,\zeta_j\rangle$ is continuous,
  it follows, from the moreover part of Lemma \ref{lem:omnibus}
   that 
  $\langle \psi,\zeta_j \rangle K_{\omega_\ell} m_\ell$
  is in $L^2(\mu)$.

   In Section \ref{sec:more-test-functions} we show that there
   is a $0<\rho<1$ and a $C$ such that  for all $j$,
   $|\langle \psi,\zeta_j\rangle | < C\rho^{|j|}$
   (see also Lemma \ref{lem:C-rho}). Note also,
 \begin{equation*}
   \psi(z) = \sum_j \langle \psi,\zeta_j\rangle \zeta_j(z).
 \end{equation*}

\begin{proof}
  Choose $C$ and $\rho$ as above. It follows from
  Lemma \ref{lem:tau-rep} that
 \begin{equation*}
    \| \langle \psi,\zeta_j \rangle F\|_{L^2(\mu)} \le
       C \rho^{|j|} \|F\|
  \end{equation*}
  and thus the sum on the right hand side of equation \eqref{eq:Z1F}
  converges. 

   Because simple
   functions are dense in $L^2(\mu)$
   by item (i) of Lemma \ref{lem:omnibus},
   it suffices to prove the result assuming $F= K_{\omega}m \otimes e_p$,
   for a Borel set $\omega$.  Given $z\in\mathbb A$
   and a (very) simple function $F^\prime= K_{\omega^\prime}m^\prime \otimes e_p$
 \begin{equation*}
  \begin{split}
    \langle Z (f\otimes F\otimes e_p),
        k_z\otimes F^\prime \otimes e_p \rangle
     = & \langle (f\otimes F\otimes e_p),
         k_z \otimes \rho(E(z))^* (F^\prime \otimes e_p)
            \rangle_{H^2(k)\otimes [L^2(\mu)\otimes \ell^2]} \\
  = & \langle (f\otimes F),
         k_z \otimes \tau(E(z))^* F^\prime  \rangle_{H^2(k)\otimes L^2(\mu)} \\
     = & \int_{\omega\cap\omega^\prime} f(z)\psi(z) \langle d\mu(\psi)m,m^\prime
             \rangle \\
     = & \sum_j f(z) \zeta_j(z) \int_{\omega\cap\omega^\prime} \langle \psi,\zeta_j
             \rangle
               \langle d\mu(\psi)m,m^\prime \rangle \\
     = & \sum_j f(z) \zeta_j(z)  \langle \langle \psi,\zeta_j\rangle
                 F, F^\prime \rangle \\
     =& \sum_j \langle f\zeta_j \otimes \langle \langle \psi,\zeta_j\rangle
                 F \otimes e_p, k_z\otimes F^\prime \otimes e_p  \rangle.
   \end{split}
 \end{equation*}
\end{proof}

\def\bW{\mathbb W}

  Returning to the transfer function $W$
  of equation \eqref{eq:the-transfer}, let $\bW=W(z)-D$.
  Before concluding this subsection, we present two key relations
  amongst $V,\bW,R,\Phi,\iota$ and $Z$.
  Define $J:\mathcal M\to H^2(k)\otimes\mathcal M$ by
\[
  Jm=\sum \zeta_j \otimes T_j^* m.
\]
  The  spectral condition $\sigma(T)\subset \mathbb A$ implies
  this sum converges and $J$ is a bounded operator. Note 
  that $(I\otimes R)J = V$.

\begin{lemma}
 \label{lem:key1}
  For $f\in H^\infty(\mathbb A)$ and $F\in L^2(\mu)\otimes \ell^2$,
\begin{equation}
 \label{eq:half-intertwine}
    J^* (I\otimes\iota^* \Phi \cW^*) Z  f\otimes F
     = T_f Y^*\Phi \cW^* F,
\end{equation}
 and
\begin{equation*}
        J^*(I\otimes\iota^* \Phi \cW^*)   f\otimes F
           =  T_f \iota^* \Phi \cW^* F.
\end{equation*}
\end{lemma}


\begin{proof}
   First, suppose $f=\zeta^p$ for some integer $p$.
  Straightforward computation and the fact that $M$ lifts $T$ gives,
\[
  \langle \zeta^p \zeta_\ell,\zeta_{p+\ell}\rangle T_{p+\ell}
    = T_\ell T^p. 
\]
   Given $m\in\mathcal M$, 
   $\omega \in \BT$ and $h\in L^2(\mu)\otimes \ell^2$,
   let $F= K_\omega \otimes h$, 
   and compute,
\begin{equation*}
  \begin{split}
    \langle J^*(I\otimes \iota^* \Phi  \cW^*) Z
        (z^p\otimes F)&,m\rangle_{\mathcal M} \\
     = & \langle (I\otimes \iota^* \Phi  \cW^*)Z
         (z^p\otimes F),    \sum \zeta_j\otimes T_j^* m\rangle_{H^2(k)\otimes \mathcal M}\\
     = & \sum_j \langle I\otimes \cW^* Z (z^p \otimes F), \zeta_j
          \otimes \Phi^* \iota T_j^* m\rangle_{H^2(k)\otimes L^2(\mu)} \\
     =& \sum_\ell \langle [\sum_\ell  z^p \zeta_\ell \otimes \langle \psi,\zeta_\ell
          \rangle \cW^* F],
          \zeta_j \otimes \Phi^* \iota T_j^* m\rangle_{H^2(k)\otimes L^2(\mu)} \\
     =& \sum_\ell \langle \zeta^p \zeta_\ell,\zeta_{p+\ell}\rangle
           \langle \psi,\zeta_{\ell}\rangle \langle \cW^* F, \Phi^*
             \iota T_{p+\ell}^* m\rangle_{L^2(\mu)} \\
     =& \sum_\ell \langle \Phi \cW^* F, \langle \zeta_\ell,\psi\rangle
             T_\ell^* (T^*)^p m\rangle_{L^2(\Lambda)} \\
     =& \langle \Phi \cW^* F,Y(T^*)^p m\rangle_{L^2(\Lambda)} \\
     =& \langle T^p Y^*  \Phi \cW^* F,m\rangle_{\mathcal M}.
 \end{split}
\end{equation*}
    Here we have used the form of $V$ from Proposition \ref{prop:AEM}
    in the second equality; the description of $Z$
    provided by Lemma \ref{lem:Z1F} in the
    fourth; and Lemma \ref{lem:intfour}, 
    equation \eqref{eq:intfour} in the seventh.

  Now use linearity and the fact that the linear span of elements
  like $F$ is dense in $L^2(\mu)\otimes \ell^2$ to finish
  the proof of the first part of Lemma \ref{lem:key1}.

  An argument very much like the one that proved the first identity
  proves the second. 
\end{proof}

We now  use Lemma \ref{lem:key1}  to establish the following Lemma. 

\begin{lemma}
 \label{lem:pre-parlay}
    With notations  as above (and $A$, $B$ and $C$ appearing in
  the representation of the transfer function $W$),
 \begin{equation*}
   J^*(I\otimes \iota^* \Phi \cW^*)[I-Z(I\otimes A)]
     = V^*(I\otimes C).
 \end{equation*}
\end{lemma}

\begin{proof}
  For $f\in H^\infty(\mathbb A)$ and $F\in L^2(\mu)\otimes e_p$,
 \begin{equation*}
  \begin{split}
   J^*(I\otimes \iota^*\Phi \cW^*)&[I-Z(I\otimes A)](f\otimes F) \\
    =& J^*(I\otimes \iota^* \Phi\cW^*)
         [f\otimes F-Z(f\otimes AF)] \\
    =& T_f [\iota^*\Phi\cW^*  - Y^*\Phi\cW^* A]F \\
    =& T_f [R^* C]F \\
    =& V^*[f\otimes CF].
  \end{split}
 \end{equation*}
    Here both parts of Lemma \ref{lem:key1} were used in the second equality,
    equation \eqref{eq:system} (i) was used in the third,
    and Proposition \ref{prop:AEM} in the last.

   Since the linear span of elements of the form $f\otimes F$
   is dense in $H^2(k)\otimes [L^2(\mu)\otimes \ell^2]$, the result follows.
\end{proof}

 The following Lemma does the heavy lifting in the proof
 of (ad) implies (sc) in Theorem \ref{thm:main}.
  Recall $\bW= W-D$. 

\begin{lemma}
 \label{lem:parlay}
   For $m\in \mathcal M$,
 \begin{equation*}
   J^*(I\otimes \iota^* \Phi\cW^*) Z (I\otimes B)(1\otimes m)
     = V^* M_{\bW}(1\otimes m).
 \end{equation*}
\end{lemma}


\def\cZ{\mathcal Z}

\begin{proof}
   Choose a sequence  $0<t_n<1$ converging to $1$ and let
 \begin{equation*}
   \mathcal Z_n=(1-t_n)[I-t_n Z(I\otimes A)]^{-1}.
 \end{equation*}
   We claim that $\cZ_n$ converges to $0$ in the WOT.
    The first step in proving this claim is to show that  $\cZ_n$ is
    contractive which follows from the following computation in which we
    have written $S$ in place of $Z(I\otimes A)$:
\begin{equation*}
  \begin{split}
   I-\cZ_n\cZ_n^*
    =& (I-t_nS)^{-1}[ (I-t_nS)(I-t_nS)^*
          -(1-t_n)^2] (I-t_nS)^{-*} \\
    =& t_n (I-t_nS)^{-1}[ -S - S^*
          -t_n(1-SS^*)] (I-t_nS)^{-*} \\
    =& t_n (I-t_nS)^{-1}[ (I-S)(I -S^*)
          +(1-t_n) (1-SS^*)    ](I-t_nS)^{-*}. \\
  \end{split}
 \end{equation*}
   Noting that $S$ is a contraction - since both
  $Z$ and $A$ are contractions - it follows
  that $\cZ_n$ is a contraction.

   Next observe that, for given
  ${\tt f}\in H^2(k)\otimes L^2(\mu)\otimes\ell^2$,
   $z\in \mathbb A$ and $F\in L^2(\mu)\otimes \ell^2$,
 \begin{equation*}
  \begin{split}
   \langle  \cZ_n {\tt f}, k_z\otimes F\rangle
    =& (1-t_n) \sum_j\langle {\tt f}, (t_n (I\otimes A)^*Z^*)^j k_z\otimes F\rangle \\
    = & (1-t_n) \sum_j \langle {\tt f}, k_z\otimes (t_nA^* \rho(E(z))^*)^j F\rangle \\
    = & (1-t_n) \langle {\tt f}, k_z \otimes (I- t_nA^* \rho(E(z))^*)^{-1}F\rangle 
  \end{split}
 \end{equation*}
   which evidently tends to $0$ as $t_n$ tends to $1$,
   since $\|\rho(E(z)) A\|<1$.  The statement
  about WOT convergence now follows.

  Let $\bW_n = C(I-t_n\rho(E(z))A)^{-1} \rho(E(z))B$.
  Because $\bW_n$ converges pointwise boundedly 
  to $\bW$,  $M_{\bW_n}$ converges WOT boundedly  to $M_{\bW}$.

  Next, for $m,h\in\mathcal M$,
 \begin{equation*}
  \begin{split}
    \langle (I\otimes C)&(I- t_n Z(I\otimes A))^{-1} Z(I\otimes B)(1\otimes m),
         k_z\otimes h\rangle \\
    =& \langle (I-t_nZ(I\otimes A))^{-1} Z 1\otimes Bm,
         k_z\otimes C^*h \rangle \\
    =& \langle (I-t_n \rho(E(z))A)^{-1} \rho(E(z)) Bm, C^*h\rangle \\
    =& \langle \bW_n(z) m, h\rangle\\
    =& \langle (M_{\bW_n}1\otimes m)(z), h \rangle\\
    =& \langle M_{\bW_n} 1\otimes m, k_z\otimes h\rangle.
  \end{split}
 \end{equation*}
  Hence
 \begin{equation*}
  (I\otimes C)(I- t_n Z(I\otimes A))^{-1} Z (I\otimes B)(1\otimes m)
   = M_{\bW_n} m.
 \end{equation*}

  We are now in a position to complete the proof.
  Using Lemma \ref{lem:pre-parlay},
 \begin{equation*}
  \begin{split}
    V^* &M_{\bW_n}(1\otimes m)  \\
     =& V^* (I\otimes C)(I- t_n Z(I\otimes A))^{-1} Z(I\otimes B)(1\otimes m)\\
     = & J^*(I\otimes \iota^*\Phi \mathcal{W}^*)[I-Z(I\otimes A)] \times \\
        &   (I- t_n Z(I\otimes A))^{-1}Z (I\otimes B)(1\otimes m)\\
     =& J^*(I\otimes \iota^*\Phi \mathcal{W}^*)Z (1\otimes Bm) \\
        & + (t_n-1) V^*(I- t_n Z(I\otimes A))^{-1} Z(I\otimes B)(1\otimes m).
  \end{split}
 \end{equation*}
   As $n$ tends to infinity, the left hand side tends to $V^*M_{W}$ (WOT)
   and the second term on the right hand side tends to $0$ (WOT)
   completing the proof.
\end{proof}

\section{Proof of (ad) Implies (sc)}
 \label{sec:proof}
   Using the ingredients assembled in the previous section,
   the proof that (ad) implies (sc) follows readily.
   For $f\in H^\infty(\mathbb A)$ and $m\in\mathcal M$,
 \begin{equation*}
  \begin{split}
    V^* M_W (f\otimes m)
      =& V^* M_f M_W(1\otimes m) \ \ \mbox{ since } M_f, M_W \mbox{ commute} \\
      =& T_f V^* M_W (1\otimes m) \ \ \mbox{ since } V^* \mbox{ intertwines } M_f  \mbox{ and } T_f  \\
      =& T_f V^* [(1\otimes Dm)+M_{\bW}(1\otimes m)]  \\
      =& T_f V^*(1\otimes Dm) + T_f J^*(I\otimes \iota^* \Phi \mathcal W^*)Z 1\otimes Bm
                     \ \ \mbox{ from Lemma \ref{lem:parlay}}\\
    = & T_f R^*Dm + T_f J^*(I\otimes \iota^* \Phi \mathcal W^*)Z 1\otimes Bm
                     \ \ \mbox{from equation \eqref{eq:Vstar}} \\
      =& T_f R^*Dm + T_f \iota^* Y^*\Phi \mathcal W^* Bm
            \ \ \ \mbox{ using equation \eqref{eq:half-intertwine}} \\
      =& T_f[ R^* D+  Y^*\Phi \cW^* B]m\\
      =& T_f XR^* m \ \ \mbox{ using the second equation in \eqref{eq:system}}\\
      =& XT_f R^*m \\
      =& X V^* (f\otimes m) \ \ \mbox{ from Proposition \ref{prop:AEM}, equation \eqref{eq:Vstar}}.
  \end{split}
 \end{equation*}


\section{The Converse}
 \label{sec:converse}
  This section is devoted to the proof of the implication
  (sc) implies (ad) of Theorem \ref{thm:main}.
  Accordingly assume hypotheses (i), (ii), and (iii)
  and also the representation (sc) for $X$
  in Theorem \ref{thm:main} throughout
  this section.  Thus there is an $W$ with a $\Psi$-unitary
  colligation transfer function representation
\begin{equation*}
  W(z)= D+C(I-\rho(E(z))A)^{-1} \rho(E(z))B
\end{equation*}
 such that $VX^*=M_W^* V$.

 For definiteness, write
\begin{equation*}
   U =\begin{pmatrix} A & B \\ C & D \end{pmatrix}:
    \begin{matrix} \mathcal E \\ \oplus \\ \mathcal H \end{matrix}  \to
     \begin{matrix} \mathcal E \\ \oplus \\ \mathcal H \end{matrix}.
\end{equation*}

 For technical reasons, let, for $0\le r < 1$,
\begin{equation*}
  W_r(z) = D+C(I-r\rho(E(z))A)^{-1} r\rho(E(z))B.
\end{equation*}

 Like before, let
\begin{equation*}
  H_r(z)= C(I-r\rho(E(z))A)^{-1}.
\end{equation*}
 The usual computation reveals,
\begin{equation}
 \label{eq:H}
  I-W_r(z)W_r(w)^* = H_r(z)(I-r^2\rho(E(z))\rho(E(w))^*)H_r(w)^*.
\end{equation}

\def\bE{\mathbb E}

There is a spectral measure $\bE$ associated with the representation
  $\rho$. Thus $\bE:\BT\to\mathcal B(\mathcal E)$ and, in particular,
 \begin{equation*}
   \rho(E(z))\rho(E(w))^* = \int_\Psi E(z)E(w)^* d\bE(\psi)
      = \int_\Psi \psi(z)\psi(w)^* d\bE(\psi),
 \end{equation*}
  where $E(z)(\psi)=\psi(z)$ has been used.

\begin{lemma}
 \label{lem:HHstar}
  There exists a constant $\kappa>0$ so that
 \begin{equation*}
   H(z)H(w)^* \preceq \kappa k(z,w).
 \end{equation*}
 (Here the inequality is in the sense of kernels).
\end{lemma}

\begin{proof}
   Multiplying equation (\ref{eq:H}) by $k(z,w)$ gives,
 \begin{equation}
  \label{eq:H2}
   (I-W_r(z)W_r(w)^*)k(z,w) =
     H_r(z)[ \, \int k(z,w)(1-r^2\psi(z)\psi(w)^*) dE(\psi) \, ] H_r(w)^*
 \end{equation}

  On the other hand, with $b=\sqrt{q}$, since $\psi(b)=0$ we have
 \begin{equation*}
    k(z,b)(1-r^2\psi(z)\psi(b)^*) = k(z,b).
 \end{equation*}
   Thus,
 \begin{equation}
  \label{ineq:rank2}
   k(z,w)(1-r^2\psi(z)\psi(w)^*) \succeq\frac{k(z,b)k(b,w)}{k(b,b)}.
 \end{equation}
  Letting  $G(z)=\frac{k(z,b)}{\sqrt{k(b,b)}}$, combining equations
   (\ref{ineq:rank2}) and (\ref{eq:H2}),
   using
 \begin{equation*}
   k(z,w)\succeq k(z,w)(I-W_r(z)W_r(w)^*)
 \end{equation*}
   and $\bE(\mathbb T)=I$,   gives,
 \begin{equation}
  \label{eq:converse-main}
    k(z,w) \succeq  H_r(z)G(z)G(w)^*H_r(w)^*.
 \end{equation}

  The function $g(z)=\frac{1}{G}$ is analytic in a neighborhood of the
  annulus and is thus a multiplier of $H^2(k)$. In particular,
  there is an $\eta$
  so that $k(z,w) (\eta^2 - g(z)g(w)^*) \succeq 0$.
  This last inequality is more
  conveniently written as
 \begin{equation}
  \label{ineq:g}
    \eta^2  k(z,w) \succeq k(z,w)g(z)g(w)^*.
 \end{equation}
  Putting equations (\ref{eq:converse-main}) and (\ref{ineq:g})
  together yields,
 \begin{equation*}
    \eta^2  k(z,w) \succeq k(z,w)g(z)g(w)^* \succeq H_r(z)H_r(w)^*
 \end{equation*}
\end{proof}

\def\tQ{ {\tt Q}}

\def\bH{\mathbb H}

  From Lemma \ref{lem:HHstar} it follows that $H_r$ induces a bounded
  linear operator  $\bH_r^*:H^2(k) \otimes  \mathcal H \to  \mathcal E$
  of norm at most $\sqrt{\kappa}$, determined by
 \begin{equation*}
   \bH_r^* k_z\otimes e=H_r(z)^*e.
 \end{equation*}
  Hence for $0<r\le 1$,   the formula
 \begin{equation*}
   \tQ_r(\omega) = \bH_r \bE(\omega) \bH_r^*,
 \end{equation*}
  for a Borel subset $\omega$ of $\mathbb T$,
  defines a $\mathcal B(\mathcal H^2(k)\otimes \mathcal H)$-valued measure.

  Let $\tQ=\tQ_1$. Define $\mu_r(\omega)=V^* \tQ_r(\omega)V$
  \index{$\mu_r$} and let $\mu=\mu_1$. Finally,
  let $\Lambda_r(\omega)=k(T,T^*)(\mu_r(\omega))$ and
  $\Lambda = \Lambda_1$.  \index{$\Lambda_r$} 

\begin{lemma}
 \label{lem:prestim}
  For fixed $\omega$, the operators $\tQ_r(\omega)$ 
  are uniformly bounded by $\kappa$, and, for each $\omega$, the net
  $\tQ_r(\omega)$ converges WOT to $\tQ(\omega)$.

  Similarly,
 \begin{equation*}
    \mu(\omega) = V^* \tQ(\omega)V
 \end{equation*}
  defines a $\mathcal B(\mathcal M)$-valued measure on $\mathbb T$
  and $\mu_r(\omega)=V^* \tQ_r(\omega)V$ converges WOT 
  boundedly to $\mu(\omega)$.

  Finally, $\Lambda_r(\mathbb T)$ is uniformly bounded
  and $\Lambda_r(\omega)$ converges WOT to $\Lambda(\omega)$
  for each Borel set $\omega$.  
\end{lemma}

\begin{proof}
   The uniform bound on the $\tQ_r$ follows immediately
  from the fact that $\mathbb H_r$ is uniformly bounded. 
  Next, 
   as $r$ tends to $1$, 
 \[
   \langle \tQ_r(\omega) k_z\otimes e, k_w\otimes f\rangle
   = \langle \bE(\omega) H_r(z)^* e, H_r(w)^* f\rangle
   \to \langle \bE(\omega) H_1(z)^*e,H_1(w)^* f \rangle.
 \]
  Here we have used, for a fixed $z\in\mathbb A$,
  $H_r(z)$ converges in norm to $H_1(z)$. 
  Since $\tQ_r(\omega)$  is uniformly bounded and converges
  WOT to $\tQ(\omega)$ against a dense set of vectors, it converges WOT to
  $\tQ(\omega).$ 

  The spectral condition on $T$ and the fact that $\mu_r(\mathbb T)$ 
  is uniformly bounded implies $\Lambda_r(\mathbb T)$ is also
  uniformly bounded.  Since $\mu_r(\omega)$ converges WOT to $\mu(\omega)$
  it follows that $k(T,T^*)(\mu_r(\omega))$ converges WOT to
  $k(T,T^*)(\mu(\omega))$. 
\end{proof}

    To complete the proof of Theorem \ref{thm:main}
  it remains to show that $\mu$ produces
  an Agler decomposition for the pair $(T,X)$.

 \begin{lemma}
  \label{lem:Lambdaisok}
    For $0<r<1$ and each Borel set $\omega\subset \mathbb T$,
    \begin{equation*}
     \Lambda_r(\omega) = V^* M_{H_r}(I\otimes \bE(\omega))M_{H_r}^* V.
    \end{equation*}

    For any $\varphi\in H^\infty(\mathbb A)$
    of norm at most one and Borel set $\omega$,
  \begin{equation*}
    \Lambda(\omega)-T_\varphi \Lambda(\omega)T_\varphi^* \succeq 0.
  \end{equation*}
 \end{lemma}

   Using the definition of $\Lambda$, the conclusion of the
   second part of the  lemma is
 \begin{equation*}
    k(T,T^*)(\mu(\omega)-T_\varphi\mu(\omega)T_\varphi^*) \succeq 0.
  \end{equation*}
    From the definition of $\mu$ and the lifting property $VT_\psi ^*=M_\psi ^*V$,
  \begin{equation*}
     \mu(\omega)-T_\varphi\mu(\omega)T_\varphi^*
       = V^* (\tQ(\omega)-M_\varphi \tQ(\omega)M_\varphi^*)V.
  \end{equation*}

\begin{proof} 
  Let 
\[
  k_n(z,w)=\sum_{-n}^n \zeta_j(z)\zeta_j(w)^*
\]
  so that for an operator $G$, 
\[
  k_n(M,M^*)(G)= \sum_{-n}^n M_j G M_j^*,
\]
  where $M_j=M_{\zeta_j}$. 
  To prove that   $k_n(M,M^*)(\tQ_r(\omega))$ converges WOT
  to $M_{H_r} (I\otimes \bE(\omega))M_{H_r}^*,$ observe,
  for a Borel set $\omega$, that
\begin{equation}
 \label{eq:kn1}
 \begin{split}
 \langle k_n(M,M^*)(\tQ_r(\omega)) k_w\otimes e,k_z\otimes f\rangle
   = & k_n(z,w)\langle \bE(\omega) H_r(w)^* e, H_r^*(z) f\rangle \\
   \preceq  & k(z,w) \langle \bE(\omega) H_r(w)^* e, H_r^*(z) f\rangle \\
   = & \langle M_{H_r} (I\otimes \bE(\omega))M_{H_r}^* k_w\otimes e,
             k_z\otimes f \rangle \\
   \preceq & \langle M_{H_r} M_{H_r}^* k_w\otimes e,
             k_z\otimes f \rangle, 
 \end{split}
\end{equation}
  where the inequalities are in the sense of (positive semidefinite) kernels. 
 
  Since also 
  $k_n(M,M^*)(\tQ_r(\omega))$ is a bounded increasing sequence of 
  positive operators equation \eqref{eq:kn1} 
  implies that $k_n(M,M^*)(\tQ_r(\omega))$ converges WOT to 
  $M_{H_r} (I\otimes \bE(\omega))M_{H_r}^*.$  Hence,
  $V^* k_n(M,M^*)(\tQ_r(\omega))V = k_n(T,T^*)(\mu_r(\omega))$
  converges to $V^* M_{H_r} (I\otimes \bE(\omega))M_{H_r}^*V,$ proving
  the first part of the Lemma.

 Similarly, $k_n(M,M^*)(M_\varphi \tQ_r (\omega) M_\varphi^*)=
      M_\varphi k_n(M,M^*)(\tQ_r(\omega))M_\varphi^*,$ 
 converges  WOT to 
  $M_\varphi M_{H_r} (I\otimes \bE(\omega))M_{H_r}^*M_\varphi^*$.
   Thus, letting $n$ tend to infinity and using the definition of $\tQ_r$,
 \begin{equation*}
  \begin{split}
    \langle k_n(M,M^*)(\tQ_r (\omega)-&M_\varphi \tQ_r (\omega)M_\varphi^*)
     k_w\otimes e, k_z\otimes f \rangle \\
    \to & (1-\varphi(z)\varphi(w)^*) k(z,w) \langle H_r(z)\bE(\omega)H_r(w)^* e,
         f\rangle
  \end{split}
 \end{equation*}
   The kernel on the right hand side is positive semi-definite
   because it is the pointwise product of positive semi-definite kernels, 
   and thus
 \begin{equation*}
  \lim_{\mbox{\tiny WOT}}  
   k_n (M,M^*) [ \, \tQ_r (\omega) - 
        M_\varphi \tQ_r (\omega)M_\varphi^* \, ] \succeq 0.
  \end{equation*}
   Thus,
 \begin{equation*}
  \begin{split}
     0\preceq & V^* \lim_{\mbox{\tiny WOT}}k_n(M,M^*)(\tQ_r(\omega)-M_\varphi \tQ_r(\omega)M_\varphi^*)V \\
      =& k(T,T^*) [ \, V^* \tQ_r(\omega)V -
         T_\varphi V^*\tQ_r(\omega)VT_\varphi^*  \, ]. 
\end{split}
\end{equation*}
 Finally, letting $r$ tend to $1$ on the right hand side above 
  and applying Lemmas \ref{lem:prestim} and \ref{lem:overk} gives
\[
     0\preceq  k(T,T^*)[\mu(\omega)-T_\varphi \mu(\omega)T_\varphi^*]. 
\]
\end{proof}

  It remains to verify the condition of equation \eqref{eq:key-mu}.
  The argument is an elaboration on the proof of the preceding
  lemma, making use of the approximations $H_r$ and the
  related operators $\bH_r$ and $M_{H_r}$.

  We break the proof 
  into several steps as outlined in the Lemma below.

\begin{lemma}
 \label{lem:stim}
  With notations as above:
 \begin{itemize}
  \item[(i)]  For $0<r<1$ and each Borel set $\omega\subset \mathbb T$,
    \begin{equation*}
     \Lambda_r(\omega) = V^* M_{H_r}(I\otimes \bE(\omega))M_{H_r}^* V
    \end{equation*}
      and converges WOT to $\Lambda(\omega)$.

  \item[(ii)] There is a constant $C_*$ such that
    $\|\Lambda_r(\mathbb T)\| \le C_*^2$ for all $0<r<1.$
  In particular, $\|M_{H_r}^* V\|\le C_*$ independent of $r$. 

  \item[(iii)] There is a bounded operator $\Gamma$
    on $H^2(k)\otimes \mathcal E$  determined by
   \begin{equation*}
    \langle \Gamma k_w\otimes f,k_z\otimes g\rangle
        = k(z,w) \int \psi(z)\psi(w)^* d\, \langle \bE(\psi)f,g\rangle.
   \end{equation*}

  \item[(iv)] If
    $(\omega_j)$ is a Borel partition of $\mathbb T$
    of diameter at most $\epsilon>0,$ then for
    any choice of points $s_j\in\omega_j$,
   \begin{equation*}
      \epsilon > \| \Gamma -\sum M_{s_j} M_{s_j}^* \otimes \bE(\omega_j) \|.
   \end{equation*}

   \item[(v)] The identity
  \begin{equation*}
     k(T,T^*)(\int T_\psi d\, \mu(\psi) T_\psi^*)
        = \int T_\psi d\, \Lambda(\psi) T_\psi^*
  \end{equation*}
   holds.
 \end{itemize}
\end{lemma}

  Note, in item (iv) the identification of $\mathbb T$ with $\Psi$
  is in force so that $|s-t|<\epsilon$ means
  $\| M_s-M_t\|=\|s-t\|_\infty<\epsilon.$

 \begin{proof}[Proof of Lemma \ref{lem:stim}]
   The first part of item (i) is
   part of Lemma \ref{lem:prestim}.  The description
  of $\Lambda_r$ in terms of $M_{H_r}$ and $\mathbb E$
  is the first part of Lemma \ref{lem:Lambdaisok}.

  To prove item (ii), note that $\Lambda_r(\mathbb T)$ is
  uniformly bounded by Lemma \ref{lem:prestim},
  so there is a $C_*$. The bound on $M_{H_r}^*$
  follows from this bound on $\Lambda_r(\mathbb T)$
  and the representation of $\Lambda_r$ in item (i). 

  Item (iii) is a consequence of the fact that,
  as kernels, $k(z,w)\psi(z)\psi(w)^* \preceq k(z,w)$.

  To prove item (iv), first note 
 if $(\omega_j)_j$ is a partition of $\mathbb T$, then
 \begin{equation*}
   \sum \Gamma(\omega_j)=\Gamma.
 \end{equation*}
  Next observe 
   that if $s,t\in\Psi$ and $|s-t|<\epsilon,$
  then, since also  $\| M_s\|,\|M_t\|=1,$ we have
  $\|M_s M_s^* -M_t M_t^*\|<\epsilon$. 
 
  To finish the proof of item (iv), choose any partition $(\omega_j)$
  of $\Psi=\mathbb T$ of width at most $\epsilon>0$. Thus,
  if $s,t\in \omega_j$, then $\|M_s-M_t\|<\epsilon;$
  i.e., the sup norm of the difference of the functions
   $s,t:\mathbb A\to\mathbb D$ is less than $\epsilon$.
   Thus, if $s_j,t_j\in\omega_j$, then
 \begin{equation*}
   \| \sum M_{s_j}M_{s_j}^* \otimes \bE(\omega_j) -
     \sum M_{t_j}M_{t_j}^*\otimes \bE(\omega_j)\| \le 2\epsilon.
 \end{equation*}

   Consequently,  choosing a sequence of partitions
   such that the width of the partitions
   tends to zero,  the corresponding Riemann sums form a norm Cauchy sequence
   and thus converge to some operator. At the same time,
   this sequence  converges
   WOT to $\Gamma$, since
 \begin{equation*}
  \langle  \sum M_{s_j} M_{s_j}^* \otimes \bE(\omega_j) k_w\otimes f,
     k_z\otimes g\rangle
     = \sum_j s_j(z)s_j(w)^* k(z,w) \langle d\, \bE(\omega_j)f,g\rangle.
 \end{equation*}
  Thus the sequence of Riemann sums converges in norm to $\Gamma$.
  Comparing any Riemann sum whose partition has width at most
  $\epsilon>0$ with an appropriate term of the  sequence just
  constructed completes the proof of (iv).

   From Lemma \ref{lem:Y-bounded} and Remark \ref{rem:YstarY}, the
   Riemann sums $\Delta(P,S,\mu)$ and
   $\Delta(P,S,\Lambda)$ 
   converge WOT to $\int T_\psi d\, \mu(\psi) T_\psi^*$
   and  $\int T_\psi d\, \Lambda(\psi) T_\psi^*$
   respectively. Hence the net  $k(T,T^*)(\Delta(P,S,\mu))$ converges
   to the RHS of item (v). On the other hand,
   we have $k(T,T^*)(\Delta(P,S,\mu)) = \Delta(P,S,\Lambda).$
   Hence $k(T,T^*)(\Delta(P,S,\mu))$ converges WOT to both
   the right and left hand side of  (v)
   and the result follows.
 \end{proof}

  Using Lemma \ref{lem:stim}, the proof that
  (sc) implies equation \eqref{eq:key-mu}, and hence the converse
  of Theorem \ref{thm:main}, 
  proceeds as follows. From Lemma \ref{lem:stim} and the
  representation
\begin{equation*}
   (I-W_r(z)W_r(w)^*)k(z,w) =
      H_r(z) [\int_\Psi (1-r^2\psi(z)\psi(w)^*)k(z,w) d\bE(\psi) ] H_r(w)^*
\end{equation*}
  it follows that
 \begin{equation*}
  V^* (I-M_{W_r} M_{W_r}^*)V
   = V^* M_{H_r} [I-r^2 \Gamma] M_{H_r}^* V.
 \end{equation*}
   The left hand side converges WOT to $I-XX^*$
  (because $M_W^*V=V M_W^*$)  and thus so does
 \begin{equation*}
    V^* M_{H_r} (I- \Gamma) M_{H_r}^*V.
 \end{equation*}
   Since, by item (i) of Lemma \ref{lem:stim} with $\omega=\mathbb T$,
   $V^* M_{H_r}M_{H_r}^* V$ converges WOT to $\Lambda(\mathbb T)$,
   it follows that
 \begin{equation}
  \label{eq:conv-toZ}
   V^* M_{H_r} \Gamma M_{H_r}^* V \to  \Lambda(\mathbb T) -I +XX^*
 \end{equation}
  WOT.

  Fix a vector $m\in\mathcal M$. 
  Given $\epsilon>0$,  choose, using (iv) of Lemma \ref{lem:stim}
  and Lemma \ref{lem:Y-bounded} respectively,
  a tagged partition $(P,S)$ such that both,
\begin{equation}
 \label{eq:conv-P}
 \begin{split}
   \frac{1}{\|m\|+1}\epsilon > &
        \| \sum M_{s_j}M_{s_j}^* \otimes \bE(\omega_j) - \Gamma \| \\
   \epsilon > & |\langle \int T_\psi d\, \Lambda(\psi)T_\psi^* m,m\rangle
     - \sum \langle T_{s_j}  \Lambda(\omega_j) T_{s_j}^* m,m \rangle |.
 \end{split}
\end{equation}

  Using item (i) of Lemma \ref{lem:stim} and equation \eqref{eq:conv-toZ}
  respectively,  choose $0<r_0<1$ (depending upon $(P,S)$) such that
  for $ r_0\le r<1$,
 \begin{equation}
  \label{eq:conv-s}
   \begin{split}
     \epsilon >
      & |\sum \langle T_{s_j}  \Lambda(\omega_j) T_{s_j}^* m,m \rangle
       - \sum \langle T_{s_j}  \Lambda_r(\omega_j) T_{s_j}^* m,m \rangle | \\
    \epsilon > & |\langle V^* M_{H_r} \Gamma M_{H_r}^* Vm,m\rangle
                  -\langle (\Lambda(\mathbb T)-I+XX^*)m,m \rangle|.
   \end{split}
 \end{equation}
  Note that combining the first inequality in equation
  \eqref{eq:conv-P} with item (ii) of Lemma \ref{lem:stim} gives,
 \begin{equation}
  \label{eq:conv-s-P}
   \begin{split}
    | \langle V^* M_{H_r} \sum M_{s_j}M_{s_j}^* & \otimes \bE(\omega_j)
        M_{H_r}^*V m,m\rangle - \langle V^* M_{H_r} \Gamma M_{H_r}^* Vm,m\rangle |\\
     =  & |\langle V^*M_{H_r} [\sum M_{s_j}M_{s_j^*}\otimes \bE(\omega_j)-\Gamma]
                M_{H_r}^* V m, m\rangle | 
     <  C_*^2  \, \epsilon.
   \end{split}
  \end{equation}
   Similarly, observe that
  \begin{equation}
   \label{eq:conv-T-M}
    \begin{split}
    \sum \langle T_{s_j}  \Lambda_r(\omega_j) T_{s_j}^* m,m \rangle
     =&\sum \langle T_{s_j} V^* M_{H_r} (I\otimes \bE(\omega_j))
        M_{H_r}^* V T_{s_j}^* m, m\rangle \\
     =& \sum \langle V^* M_{H_r} M_{s_j}(I\otimes \bE(\omega_j))
       M_{s_j}^* M_{H_r}^* Vm, m \rangle \\
     =& \langle V^* M_{H_r}[ \sum M_{s_j}(I\otimes \bE(\omega_j))
       M_{s_j}^*] M_{H_r}^* Vm, m \rangle.
    \end{split}
  \end{equation}

   Putting it all together, it follows from \eqref{eq:conv-s},
   \eqref{eq:conv-s-P}, and  \eqref{eq:conv-T-M} that
 \begin{equation*}
  \begin{split}
    |\langle \int T_\psi d\, & \Lambda(\psi) T_\psi^* m,m\rangle
      - \langle (\Lambda(\mathbb T)-I+XX^*)m,m\rangle | \\
     \le & |\langle \int T_\psi d\, \Lambda(\psi) T_\psi^* m,m\rangle
         - \sum \langle T_{s_j}  \Lambda(\omega_j) T_{s_j}^* m,m \rangle | \\
      & + | \sum \langle T_{s_j}  \Lambda(\omega_j) T_{s_j}^* m,m \rangle
        - \sum \langle T_{s_j}  \Lambda_r(\omega_j) T_{s_j}^* m,m \rangle |\\
        & + | \langle V^* M_{H_r}[ \sum M_{s_j}(I\otimes \bE(\omega_j))
            M_{s_j}^* - \Gamma] M_{H_r}^*
             Vm,m\rangle | \\
      & + |  \langle V^* M_{H_r} \Gamma M_{H_r}^* Vm,m\rangle - 
                   \langle (\Lambda(\mathbb T)-I+XX^*)m,m\rangle |\\
   < & \epsilon + \epsilon + C_*\epsilon + \epsilon.
  \end{split}
 \end{equation*}
   Thus,
 \begin{equation*}
    I-XX^* = \Lambda(\mathbb T)-\int T_\psi d\, \Lambda(\psi) T_\psi^*.
 \end{equation*}
   An application of item (v) of Lemma \ref{lem:stim}
   completes the proof.

\section{Details on the kernel}
 \label{sec:morek}
  This section gives the details on the
  basic facts about our kernel $k$. It requires a digression into
  theta functions much of which is borrowed from \cite{McShen}.

  Begin by recalling the theta function
\begin{equation*}
   \vartheta_1(x)=\vartheta_1(x,q)=2q^{\frac14} \sin(x) \Pi_{n=1}^\infty
     (1-q^{2n}) (1-q^{2n}e^{2ix}) (1-q^{2n}e^{-2ix}),
\end{equation*}
  and the Jordan-Kronecker function
\begin{equation*}
   f(\alpha,p)= \sum_{n=-\infty}^{\infty} \frac{\alpha^n}{1-pq^{2n}}.
\end{equation*}
  It is well known that these functions are related by
\begin{equation*}
   f(\alpha,p)= C \frac{\vartheta_1(x+y)}{\vartheta_1(x)\vartheta_1(y)},
\end{equation*}
  where $x$ and $y$ are chosen so that $\alpha=e^{2ix}$ and $p=e^{2iy}$
  and $C$ is a constant (independent of $x,y$).

  Replacing $p$ with $-t$ and thus $y$ with $y+\frac{\pi}{2}$ 
  and letting $\alpha=zw^*$ gives,
\[
  k(z,w;t)= C\frac{\vartheta_1(x+y+\frac{\pi}{2})}
     {\vartheta_1(x)\vartheta_1(y+\frac{\pi}{2})}
\]
 From its product expansion, it is evident that the zeros of 
    $\vartheta_1$ are $q^{2m}=e^{2ix}$ for integers $m$ and thus 
  $k(z,w;t)=0$ if
 and only if $tzw^* = -q^{2m}$ for some integer $m$.
 Thus, unless $t=q^{2\ell}$ for some $\ell$, there exists points $z,w\in\mathbb A$
  such that $k(z,w;t)=0$. 

  We are interested in the case $t=1$ ($p=-1$ and $y=0$ above) which gives,
\begin{equation*}
   k(z,w;1)=k(z,w)
  = C\frac{\vartheta_1(x+\frac{\pi}{2})}{\vartheta_1(x)\vartheta_1(\frac{\pi}{2})}.
\end{equation*}
   In particular, $k(z,w)$ vanishes if and only if $zw^*=-q^{2m}$. In particular,
   $k(z,w)$ does not vanish for both $z$ and $w$ in the annulus, and further
   for each fixed $w\in\mathbb A$,
  as a function of $z$, the kernel
  $k(z,w)$ extends beyond the annulus to a meromorphic function.

   If $zw^*=e^{2ix}$, then $-zw^* = e^{2i(x+\frac{\pi}{2})}$ and therefore,
\begin{equation*}
 \begin{split}
   k(z,-w)= & \ C\frac{\vartheta_1(x+\pi)}{\vartheta_1(x+\frac{\pi}{2})\vartheta_1(\frac{\pi}{2})} \\
          =& - \ C\frac{\vartheta_1(x)}{\vartheta_1(x+\frac{\pi}{2})\vartheta_1(\frac{\pi}{2})} \\
          =& \ C^\prime \frac{1}{k(z,w)},
 \end{split}
\end{equation*}
  where $C^\prime =\theta_1(\frac{\pi}{2})^{-2}.$
  It is evident that $C^\prime >0$.

\section{Details on the Test Functions}
 \label{sec:more-test-functions}
  Generally the minimal inner functions on a multiply connected
  domain can be constructed using the Green's functions or
  as a product of quotients of theta functions.  In the
  case of the annulus the first construction is relatively
  simple to describe, given unique solutions to the Dirichlet
  problem.

  The first step is to construct, given a point $a\in\mathbb A$,
  an analytic function with modulus one on the outer boundary $B_0$
  and constant modulus on the inner boundary $B_1$ with just one
  zero, at $a$, in $\mathbb A$.
  There is a harmonic function $w$ whose boundary values (on
  the boundary of $\mathbb A$) agree with the boundary
  values of $\log|z-a|$.  There is a constant $\beta$
  and an analytic function $f$ on $\mathbb A$ so that
\begin{equation*}
  w=\Re (f+\beta \log(|z|)).
\end{equation*}
  Here $\Re$ denotes the real part.
  Note that $\beta$ can be computed because
  a harmonic function $u=w-\beta \log(|z|)$ is the real part of
  an analytic function on $\mathbb A$ if
  and only if  the integral
  of $u$ around the outer boundary agrees with  the
  integral of $u$ around the inner boundary of $\mathbb A$; i.e.,
\begin{equation*}
    + 2\pi \beta \log(q) =\int \log|\exp(it)-a|dt -\int\log|q\exp(it)-a|dt.
\end{equation*}
  Indeed, a simple computation shows $\beta=\frac{\log(|a|)}{\log(q)}$. 
  In particular, given two points $a,b\in\mathbb A$, there
  is a function unimodular on the boundary of $\mathbb A$
  with zeros precisely $a$ and $b$ (with multiplicity if needed)
  if and only if $\log(|ab|) = q$.

\printindex

\end{document}